\documentclass[12pt]{amsart}

\usepackage{fullpage}
\usepackage{amsfonts,amssymb}
\usepackage{enumerate}
\usepackage{verbatim}
\usepackage{asymptote}

\newcommand{\C}{\mathbb{C}}
\newcommand{\N}{\mathbb{N}}
\newcommand{\R}{\mathbb{R}}
\newcommand{\Sp}{\mathbb{S}}
\newcommand{\T}{\mathbb{T}}
\newcommand{\Z}{\mathbb{Z}}

\newcommand{\calB}{{\mathcal {B}}}

\newcommand{\calS}{{\mathcal {S}}}

\newcommand{\calH}{{\mathcal {H}}}

\newcommand{\F}{\mathbf F}
\newcommand{\B}{\mathbf B}

\newcommand{\supp}{\operatorname{supp}}

\newcommand{\ov}{\overline}
\newcommand{\ch}{\mathbf 1}

\newcommand{\ve}{\varepsilon}
\newcommand{\vp}{\varphi}

\newcommand{\lan}{\langle}
\newcommand{\ran}{\rangle}

\newcommand{\proj}{{\rm proj}}

\def\Koniec{\hbox to\hsize{\hfil$\diamond$}}
\def\th{\vartheta}
\def\1{\mathbf 1}

\newtheorem{theorem}{Theorem}[section]
\newtheorem{corollary}[theorem]{Corollary}
\newtheorem{lemma}[theorem]{Lemma}
\newtheorem{proposition}[theorem]{Proposition}

\theoremstyle{definition}
\newtheorem{definition}[theorem]{Definition}
\theoremstyle{remark}
\newtheorem{remark}{Remark}[section]

\numberwithin{equation}{section}

\begin{document}

\title{Marcinkiewicz averages of smooth orthogonal projections on sphere}

\author{Marcin Bownik}
\address{Department of Mathematics, University of Oregon, Eugene, OR 97403--1222, USA}
\email{mbownik@uoregon.edu}

\author{Karol Dziedziul}
\address{
Faculty of Applied Mathematics,
Gda\'nsk University of Technology,
ul. G. Narutowicza 11/12,
80-952 Gda\'nsk, Poland}

\email{karol.dziedziul@pg.edu.pl}

\author{Anna Kamont}
\address{
Institute of Mathematics, Polish Academy of Sciences, ul. Abrahama 18, 80--825 Sopot, Poland}

\email{anna.kamont@impan.pl}

\begin{abstract} We construct a single smooth orthogonal projection with desired localization whose average under a group action yields the decomposition of the identity operator. For any full rank lattice $\Gamma \subset\R^d$, a smooth projection is localized in a neighborhood of an arbitrary precompact fundamental domain $\R^d/\Gamma$. We also show the existence of a highly localized smooth orthogonal projection, whose Marcinkiewicz average under the action of $SO(d)$, is a multiple of the identity on $L^2(\Sp^{d-1})$. As an application we construct highly localized continuous Parseval frames on the sphere.
\end{abstract}


\thanks{The first author was partially supported by the NSF grant DMS-1956395.}


\date{\today}

\maketitle

\section{Introduction}

Smooth projections on the real line were introduced in a systematic way by Auscher, Weiss, and Wickerhauser \cite{AWW} in their study of local sine and cosine bases of Coifman and Meyer \cite{CM} and in the construction of smooth wavelet bases in $L^2(\R)$, see also \cite{HW}. While the standard procedure of tensoring can be used to extend their construction to the Euclidean space $\mathbb R^d$, an extension of smooth projections to the sphere $\Sp^{d-1}$ was shown by the first two authors in \cite{BD}. A general construction of smooth orthogonal projections on a Riemannian manifold $M$, which is based partly on the Morse theory, was recently developed by the authors \cite{BDK}. We have shown that the identity operator on $M$ can be decomposed as a sum of smooth orthogonal projections subordinate to an open cover of $M$. This result, which is an operator analogue of the ubiquitous smooth partition of unity 
of a manifold, can be used to construct Parseval wavelet frames on Riemannian manifolds \cite{BDK2}.

The goal of this paper is to show the existence of a single smooth projection with desired localization properties and whose average under a group action yields the decomposition of the identity operator. We show such result in two settings. In the setting of $\R^d$ we construct a smooth orthogonal decompositions of identity on $L^2(\R^d)$, generated by translates of a single projection, which is localized in a neighborhood of an arbitrary precompact fundamental domain. In other words, a characteristic function of a fundamental domain $K$ of $\R^d$ under the action of a full lattice $\Gamma \subset \R^d$, can be smoothed out to a projection Hestenes operator localized in a neighborhood of $K$. 
In the setting of the sphere $\Sp^{d-1}$ we show the existence of a single smooth orthogonal projection which has arbitrarily small support and whose Marcinkiewicz average under the action of $SO(d)$ is a multiple of the identity on $L^2(\Sp^{d-1})$. We also show that the same decomposition works for other function spaces on $\Sp^{d-1}$. More precisely, we have the following theorem.

\begin{theorem}\label{easyth}
Let $\calB$ be a ball in $\Sp^{d-1}$. Let $\mu=\mu_d$ is a normalized Haar measure on $SO(d)$. For $b\in SO(d)$ and a function $f$ on $\Sp^{d-1}$, let $T_bf(x)=f(b^{-1}x)$. Then the following holds.

(i) 
There exist a Hestenes operator $P_\calB$  localized on $\calB$ such that $P_\calB:L^2(\Sp^{d-1}) \to L^2(\Sp^{d-1})$ is an orthogonal projection  
and for all $f\in L^2(\Sp^{d-1})$
\begin{equation}\label{easyth2}
\int_{SO(d)} T_b\circ P_\calB\circ T_{b^{-1}} (f) d\mu(b) =c(P_\calB) f,
\end{equation}
where $c(P_\calB)$ is a constant depending on $P_\calB$; the integral in \eqref{easyth2} is understood as Bochner integral with values in $L^2(\Sp^{d-1})$. 

(ii) Let $X$ be  one of the following quasi-Banach spaces: 
 Triebel-Lizorkin space $\F^s_{p,q}(\Sp^{d-1})$, $0< p,q<\infty$, $s\in \R$,
  Besov space $\B^s_{p,q}(\Sp^{d-1})$, $ 0<p,q< \infty$, $s\in\R$, 
 Sobolev space $W^k_p(\Sp^{d-1})$, $1\le p<\infty$, and $C^k(\Sp^{d-1})$, $k\geq 0$. 
Then the formula  \eqref{easyth2} holds for all $f\in X$ with the integral in \eqref{easyth2}  understood as the Pettis integral. In the case $X$ is Banach space the integral is Bochner integral. 
\end{theorem}

In the literature there are two approaches to construct a continuous frame on $L^2(\Sp^{d-1})$. A purely group-theoretical construction started with a paper   Antoine and Vandergheynst \cite{Antoine}. 
A continuous wavelet on the sphere is a function $g\in L^2(\Sp^{d-1})$ such that the family
\[
\{ T_b D_t g :(b,t)\in  SO(d)\times \R_+\},
\]
is a continuous frame in $L^2(\Sp^{d-1})$, where $D_t$ is a dilation operator. The existence of such $g$ is highly non-trivial already for $\Sp^2$ and was investigated by \cite{Dahlke}.
The second approach involves a more general wavelet transform, where dilations are replaced by a family of functions $\{g_t: t>0\}\subset L^2(\Sp^{d-1})$. This family generates a continuous Parseval frame if wavelet transform
\[
W(f)(b,t)=\int_{\Sp^{d-1}} g_t (b^{-1}x) f(x) d\sigma_{d-1}(x), \quad W:L^2(\Sp^{d-1}) \to L^2( SO(d)\times \R_+, d\mu_d  da/a)
\]
is an isometric isomorphism, see \cite[Theorem III.1]{Hol}. If functions $g_t$ are zonal, that is 
$g_t(x)=\tilde{g}_t(\langle y , x \rangle)$ for some $y \in \Sp^{d-1}$, then wavelet transform takes a simplified form
\[
W(f)(\xi,t)=\int_{\Sp^{d-1}} \tilde{g}_t(\langle \xi,  x \rangle) f(x) d\sigma_{d-1}(x),\quad  W:L^2(\Sp^{d-1}) \to L^2( \Sp^{d-1}\times \R_+, d\sigma_{d-1}da/a)
\]
Such transforms were studied for $d=3$ in \cite{Freeden, Freeden2}.
For more general weights on $\R^+$, see \cite[Theorem 3.3]{Iglewska}. A general approach to construct continuous frame wavelets on compact manifold was done by Geller and Mayeli \cite{Geller}.

As an application of Theorem \ref{easyth} we construct a highly localized continuous frame in $L^2(\Sp^{d-1})$. Unlike earlier constructions of continuous wavelet frames on $\Sp^{d-1}$, the ``dilation" space $\R^+$ is replaced by a parameter space $X$ of a local continuous Parseval frame. Moreover, our continuous wavelet frames have arbitrarily small support. A recent solution of discretization problem by Freeman and Speegle \cite{FS} yields a discrete frame on sphere \cite{Bownik, Fornasier}.

The paper is organized as follows. In Section \ref{S2} we recall the definition of Hestenes operators and Marcinkiewicz averages. In Section \ref{S3} we show the existence of a smooth orthogonal decomposition of identity on $L^2(\R^d)$ generated by a single projection. In Sections \ref{S4}--\ref{S5} we study Marcinkiewicz averages of smooth orthogonal projections on the sphere. This culminates in the proof of the first part of Theorem \ref{easyth} in Section \ref{S6}. The proof of the second part of Theorem \ref{easyth} dealing with function spaces is shown in Section \ref{S7}. In Section \ref{S8} we construct a continuous Parseval frame on the sphere.

\section{Preliminaries}\label{S2}

We recall the definition of Hestenes operators \cite[Definition 1.1]{BDK} and their localization \cite[Definition 2.1]{BDK}. Let $M$ be a Riemannian manifold. 
In this paper we only consider $M=\R^d$ or $\Sp^{d-1}$.

\begin{definition}\label{H}
 Let $\Phi:V \to V'$ be a $C^\infty$ diffeomorphism between two open subsets $V, V' \subset M$. Let $\vp: M \to \R$ be a compactly supported
$C^\infty$ function such that 
\[
\supp \vp =\ov{\{x\in M : \vp(x) \ne 0\}} \subset V.
\]
We define a simple $H$-operator $H_{\vp,\Phi,V}$ acting on a  function $f:M \to\C$ by
\begin{equation}\label{HeHe}
H_{\vp,\Phi,V}f(x) = \begin{cases} \vp(x) f(\Phi(x)) & x\in V
\\
0 & x\in M \setminus V.
\end{cases}
\end{equation}
Let $C_0(M)$ be the space of continuous real-valued functions vanishing at infinity. Clearly, a simple $H$-operator induces a continuous linear map of the space $C_0(M)$ into itself. We define a Hestenes operator to be a finite combination of such simple $H$-operators. The space of all $H$-operators is denoted by $\mathcal H(M)$.
\end{definition}

\begin{definition}\label{localized}
 We say that an operator $T \in \mathcal H(M)$ is {\it localized} on an open set $U \subset M$, if it is a finite combination of simple $H$-operators $H_{\vp,\Phi,V}$ satisfying $V\subset U$ and $\Phi(V) \subset U$.
\end{definition}

By  \cite[Lemma 2.1]{BDK} an operator $T$ is localized on $U \subset M$ if and only if there exists a compact set $K \subset U$ such that for any $f\in C_0(M)$
\begin{align}
&\supp T f \subset K, \label{ilo1}
\\
&\supp f \cap K=\emptyset \implies Tf=0. \label{ilo2}
\end{align}

For any function $f$ on $\Sp^{d-1}$, define its rotation by $b\in SO(d)$ as
\[
T_b(f)(x)=f(b^{-1}x), \qquad x\in \Sp^{d-1}.
\]
Let $\mathcal D=\C^\infty(\Sp^{d-1})$ be the space of test functions. Let $\mathcal D'$ be the dual space of distributions on $\Sp^{d-1}$.

\begin{definition}\label{Marcin}
Let $X$ be a quasi Banach space on which $X'$ separates points such that:
\begin{enumerate}
\item we have continuous embeddings ${\mathcal D} \hookrightarrow X \hookrightarrow{\mathcal D}'$ and ${\mathcal D}$ dense in $X$, 
\item there is a constant $C>0$ such that for all $b\in SO(d)$ 
\[
\|T_b\|_{X\to X}\leq C,
\]
\end{enumerate}
Let $P:X \to X$ be a bounded linear operator.
We define the Marcinkiewicz average $\calS(P)$ as the Pettis integral
\begin{equation}\label{mck}
\calS (P)(f)= \int_{SO(d)}  T_b\circ P\circ T_{b^{-1}}(f) d\mu(b),\qquad  f\in X,
\end{equation}
where $\mu=\mu_d$ is the normalized Haar measure on $SO(d)$. 
\end{definition} 

\begin{remark} Marcinkiewicz has considered such averages in the context of interpolation of trigonometric polynomials, see \cite[Theorem 8.7 in Ch. X]{Z}. 
In Section \ref{S7} we will show that the mapping
\[
SO(d) \ni b \mapsto T_b\circ P\circ T_{b^{-1}}(f) \in X
\]
is continuous.
Hence,  in the case $X$ is a Banach, \eqref{mck} exists as the Bochner integral by \cite[Theorem II.2]{Diestel}. In particular, when $X=C(\Sp^{d-1})$ we can interpret \eqref{mck} as the Bochner integral.
\end{remark}

 \begin{lemma}\label{perturbacja}
Let $\psi\in C^\infty(\Sp^{d-1})$. Let $M_\psi$ be a multiplication operator, i.e. $M_\psi (f)=\psi f$. 
Then for $f\in C(\Sp^{d-1})$ and $\xi\in \Sp^{d-1}$,
\[
\calS(M_\psi)f(\xi)=C(\psi) f(\xi),
\qquad\text{where }
C(\psi)=\int_{\Sp^{d-1}} \psi(\xi) d\sigma(\xi).
\]
\end{lemma}

\begin{proof}
Note that 
\[
\calS(M_\psi)f(\xi)=f(\xi)\int_{SO(d)} \psi(b^{-1}(\xi))d\mu(b).
\]
Letting $G=SO(d)$ and
$H=\{b\in SO(d): b(\1)=\1 \}\subset SO(d)$, we have
 $G/H=\Sp^{d-1}$. Hence, by \cite[Theorem 2.51]{Folland}
\[
C(\psi)=\int_{SO(d)} \psi(b^{-1}(\xi))d\mu(b)=\int_{\Sp^{d-1}} \psi(\xi) d\sigma(\xi). \qedhere
\]
\end{proof}

  \subsection{Marcinkiewicz averages in $L^2(\Sp^{d-1})$}
Let $\calH^d_n$ be the linear space of real harmonic
polynomials, homogeneous of degree $n$, on $\R^d$. Spherical harmonics are the restrictions of elements in $\calH^d_n$ to the unit sphere, see  \cite[Definition 1.1.1]{DaiXu}.
Let 
\[
\proj_n:L^2(\Sp^{d-1}) \to \calH^d_n
\]
denote the orthogonal projection. Since $L^2(\Sp^{d-1})$ is the orthogonal sum of the spaces $\calH^d_n$, $n=0,1,\ldots$, we can define multiplier operator with respect to spherical harmonic expansions \cite[Definition 2.2.7]{DaiXu}.

\begin{definition}
A linear operator $T: L^2(\Sp^{d-1}) \to L^2(\Sp^{d-1})$  is called a
multiplier operator if there exists a bounded sequence $\{ \lambda_n \}_{n\geq 0}$ of real numbers such that
for all $f\in L^2(\Sp^{d-1})$ and all $n\geq 0$
\[
\proj_n (T f) = \lambda_n \proj_n f.
\]
\end{definition}

Conversely, any bounded sequence $\{ \lambda_n \}_{n\geq 0}$ defines a multiplier operator on $L^2(\Sp^{d-1})$
\[
Tf = \sum_{n=0}^\infty \lambda_n \proj_n f \qquad\text{for }f \in L^2(\Sp^{d-1}).
\]
The following result characterizes Marcinkiewicz averages on the sphere, see \cite[Proposition 2.2.9]{DaiXu}.

\begin{theorem}
Let $T: L^2(\Sp^{d-1}) \to L^2(\Sp^{d-1})$ be a bounded linear operator. The following are equivalent:
\begin{enumerate}[(i)]
\item
 $T$ is a multiplier operator.
\item $T$ is invariant under the group of rotations, that is, $T T_b = T_b T$ for all $b\in SO(d)$,
\item $\calS(T)=T$.
\end{enumerate}
\end{theorem}

 \section{Orthogonal decomposition by shifts of a localized projection}\label{S3}

In this section we will show the existence of smooth orthogonal decompositions of identity on $L^2(\R^d)$, which are generated by translates of a single projection, which is localized in a neighborhood of an arbitrary precompact fundamental domain.  

Let $s\in C^\infty(\R)$ be a real-valued function such that
\begin{equation}\label{sfun}
\begin{aligned}
\supp s\subset [-\delta,+\infty) &\qquad\text{for some }\delta>0,
\\
s^2(t)+s^2(-t)=1 & \qquad \text{for all } t\in \R.
\end{aligned}
\end{equation}
Following \cite[eq. (2.9)]{BD} and \cite[eqs. (3.3) and (3.4) in Ch. 1]{HW}, for a given $\alpha<\beta$ and $\delta<\frac{\beta-\alpha}{2}$, we define an orthogonal projection $P_{[\alpha,\beta]}:L^2(\R)\to L^2(\R)$ by
\begin{equation}\label{exp}
P_{[\alpha,\beta]}f(t)=
\begin{cases}
0 & t<\alpha-\delta,
\\
s^2(t-\alpha)f(t)+ s(t-\alpha)s(\alpha-t)f(2\alpha-t) &
t \in [\alpha-\delta,\alpha+\delta],
\\
f(t) & t\in (\alpha+\delta,\beta-\delta),
\\
s^2(\beta-t) f(t) - s(t-\beta)s(\beta-t)f(2\beta-t) &
t\in [\beta-\delta,\beta+\delta],
\\
0 & t> \beta + \delta.
\end{cases}
\end{equation} 
Let $T_k$ be the translation operator by $k\in \R$ given by $T_kf(x)=f(x-k)$.
Note that for all functions $f$ and all $\alpha<\beta$, $k\in \R$, we have
\begin{equation}\label{translacja}
P_{[\alpha+k,\beta+k]}(f)(x)= (T_k  P_{[\alpha,\beta]} T_{-k}) f(x).
\end{equation}
By \cite[Theorem 1.3.15]{HW} we have the following sum rule for projections on adjacent intervals corresponding to the same $\delta<\min((\beta-\alpha)/2, (\gamma-\beta)/2)$,
\begin{equation}\label{summing}
P_{[\alpha,\beta]}+P_{[\beta,\gamma]}=P_{[\alpha,\gamma]}.
\end{equation}

Let $K \subset \R^d$ be a  fundamental domain of $\R^d/\Gamma$, where $\Gamma \subset \R^d$ is a full rank lattice. That is, $\{K+\gamma: \gamma \in \Gamma\}$ is a partition of $\R^d$ modulo null sets. Define an orthogonal projection onto $L^2(K)$ by $Pf(x)=\ch_K(x) f(x)$. Then, we have a decomposition of the identity operator $\mathbf I$ on $L^2(\R^d)$,
\[
\sum_{\gamma \in \Gamma } T_\gamma P_K T_{-\gamma} = \mathbf I.
\]
The following theorem shows that there exists a smooth variant of an operator $P_K$, satisfying the same decomposition identity, which is an $H$-operator localized on a neighborhood of $K$.

\begin{theorem} \label{lat}
Let $\Gamma \subset \R^d$ be a full rank lattice. Let $K \subset \R^d$ be a precompact fundamental domain of $\R^d/\Gamma$. Then for any $\epsilon>0$, there exists a Hestenes operator $P$, which is an orthogonal projection localized on $\epsilon$-neighborhood of $K$, such that 
\begin{equation}\label{lat0}
\sum_{\gamma \in \Gamma } T_\gamma P T_{-\gamma} = \mathbf I.
\end{equation}
Here the convergence is in the strong operator topology in $L^2(\R^d)$. In particular, projections 
$T_\gamma P T_{-\gamma}$, $\gamma \in \Gamma$, are mutually orthogonal.
\end{theorem}

\begin{proof}
We will show first that it suffices to prove the theorem for the lattice $\Z^d$. Assume momentarily that Theorem \ref{lat} holds in this special case. An arbitrary full rank lattice $\Gamma \subset \R^d$ is of the form $\Gamma=M \Z^d$ for some $d\times d$ invertible matrix $M$. If $K \subset \R^d$ is a precompact fundamental domain of $\R^d/\Gamma$, then $M^{-1}(K)$ is a precompact fundamental domain of $\R^d/\Z^d$ since
\[
\{M^{-1}(K+\gamma): \gamma \in \Gamma\} =\{M^{-1}(K)+k: k \in \Z^d\}
\]
is a partition of $\R^d$ modulo null sets. Hence, for any $\epsilon>0$, there exists a Hestenes operator $P'$, which is an orthogonal projection localized on $\epsilon$-neighborhood of $M^{-1}(K)$ such that 
\[
\sum_{k\in \Z^d } T_k P' T_{-k} = \mathbf I.
\]
Define a Hestenes operator $P= D_{M^{-1}} P'  D_{M}$, where $D_M$ is a dilation operator $D_M f(x)= f(Mx)$. Since $|\det M|^{1/2}D_M$ is an isometric isomorphism of $L^2(\R^d)$ we deduce that $P$ is an orthogonal projection. Since $T_k D_M= D_M T_{Mk}$, we have
\[
\begin{aligned}
\sum_{\gamma\in \Gamma} T_\gamma P T_{-\gamma} = \sum_{k\in \Z^d} T_{Mk} D_{M^{-1}} P'  D_{M}T_{-Mk} & = \sum_{k\in \Z^d} D_{M^{-1}} T_k P'  T_{-k} D_{M} 
\\
&= D_{M^{-1}} \circ \bigg(  \sum_{k\in \Z^d} T_k P'  T_{-k} \bigg) \circ D_{M} =\mathbf I. 
\end{aligned}
\]
Since $P'$ is localized on $\epsilon$-neighborhood $U$ of $M^{-1}(K)$ we deduce that $P= D_{M^{-1}} P'  D_{M}$ is localized in $M(U)$, which is contained in $||M||\epsilon$-neighborhood of $K$. Since $\epsilon>0$ is arbitrary, this concludes the reduction step. 

Next we will show the theorem in the special case when the lattice $\Gamma=\Z^d$ and the fundamental domain is the unit cube  $K=[0,1]^d$. Let $P_{[0,1]}$ be the orthogonal projection on $L^2(\R)$, which is given by \eqref{exp}, and localized on open interval $(-\delta,1+\delta)$. Since $P_{[0,1]}$ has opposite polarities at the endpoints, by \eqref{translacja} and \eqref{summing}  we have
\begin{equation}\label{1d}
\sum_{k\in \Z } T_k P_{[0,1]} T_{-k} = \mathbf I,
\end{equation}
where the convergence is in the strong operator topology in $L^2(\R)$, see \cite[Formula (3.18) in Ch. 1]{HW}.
Define $P_K$ as the $d$-fold tensor product $P_K=P_{[0,1]}\otimes \ldots \otimes P_{[0,1]}$, see \cite[Lemma 3.1]{BD}. That is, $P_K$ is defined initially on separable functions
\[
(f_1 \otimes \ldots \otimes f_d)(x_1,\ldots,x_d) = f_1(x_1) \cdots f_d(x_d),
\qquad\text{for }x=(x_1,\ldots,x_d) \in \R^d,
\]
by
\[
P_K(f_1\otimes \ldots \otimes  f_d)=P_{[0,1]}(f_1) \otimes \ldots \otimes  P_{[0,1]}(f_d) 
\]
and then extended to a Hestenes operator on $\R^d$. Then, $P_K$ is an orthogonal projection localized on a cube $(-\delta,1+\delta)^d$. Then, using \eqref{1d} we can verify its $d$-dimensional analogue for separable functions
\begin{equation}\label{dd}
\begin{aligned}
\sum_{k\in \Z^d } T_k P_{K} T_{-k} (f_1 \otimes \ldots \otimes f_d) &= \sum_{(k_1,\ldots,k_d) \in \Z^d} T_{k_1} P_{[0,1]} T_{-k_1}(f_1)
\otimes \ldots \otimes T_{k_d} P_{[0,1]} T_{-k_d} (f_d)
\\
&=
f_1 \otimes \ldots \otimes f_d.
\end{aligned}
\end{equation}
Since linear combinations of separable functions are dense in $L^2(\R^d)$, the above formula holds for all functions in $L^2(\R^d)$. Choosing $\delta>0$ such that $\sqrt{d} \delta< \epsilon$ yields the required projection $P=P_K$ satisfying \eqref{lat0}.

By the scaling argument we obtain the same conclusion for the  lattice $\Gamma = n^{-1} \Z^d$, and the fundamental domain $n^{-1}[0,1]^d$, where $n\in \N$. That is, define a projection $P'= D_{M^{-1}} P_{[0,1]^d} D_M$, where $M=n^{-1}\mathbf I_d$ is a multiple of $d\times d$ identity matrix $\mathbf I_d$. That is, $P'$ is a Hestenes operator, which is an orthogonal projection on $L^2(\R^d)$ satisfying
\begin{equation}\label{lat5}
\sum_{k\in n^{-1}\Z^d } T_k P' T_{-k} = \mathbf I.
\end{equation}

Let $K$ be an arbitrary precompact fundamental domain of $\R^d/\Z^d$. Choose $n\in \N$ such that
\begin{equation}\label{lat7}
(\sqrt{d}+2)/n<\epsilon.
\end{equation}
Let $P'$ be a Hestenes operator, which is orthogonal projection localized on $1/n$-neighborhood of $n^{-1}[0,1]^d$ such that \eqref{lat5} holds. Let 
\begin{equation}\label{lat9}
F_0 = \{k\in n^{-1} \Z^d: (n^{-1}[0,1]^d+k) \cap K  \ne \emptyset \}.
\end{equation}
Since $K$ is a fundamental domain of $\R^d/\Z^d$ we have
\begin{equation}\label{lat11}
\bigcup_{l\in \Z^d} (l+ F_0) = n^{-1} \Z^d.
\end{equation}
We define an equivalence relation on $F_0$: $k,k' \in F_0$ are in relation if  $k-k' \in \Z^d$. Then, we choose a subset $F_1 \subset F_0$ containing exactly one representative in each equivalence class. Hence, the family $\{l+ F_1: l\in  \Z^d\}$ is a partition of the lattice $n^{-1} \Z^d$. Define a Hestenes operator
\[
P = \sum_{k\in F_1}  T_k P' T_{-k}.
\]
Since projections $T_k P' T_{-k}$, $k\in n^{-1}\Z^d$, are mutually orthogonal, $P$ is also an orthogonal projection on $L^2(\R^d)$. Since the operator $T_k P' T_{-k}$ is localized on $1/n$-neighborhood of the cube $n^{-1}[0,1]^d+k$, whose diameter is $<\epsilon$ by \eqref{lat7}, we deduce by \eqref{lat9} that $P$ is localized on $\epsilon$-neighborhood of $K$. Combining \eqref{lat5} with the fact that  $\{l+ F_1: l\in  \Z^d\}$ is a partition of the lattice $n^{-1} \Z^d$ yields
\[
\sum_{l\in \Z^d } T_l P T_{-l}
=
 \sum_{l\in \Z^d } \sum_{k\in F_1 } T_{k+l} P' T_{-(k+l)}
 =\mathbf I.
 \] 
The convergence is in the strong operator topology in $L^2(\R^d)$.
\end{proof}

The following example illustrates Theorem \ref{lat} by an example. Let $K$ be a hexagon with the vertices: 
\[
p_1=(1,0), \ p_2=(1/2,\sqrt{3}/2), \ p_3=(-1/2,\sqrt{3}/2),\ p_4=-p_1,\ p_5=-p_2,\ p_6=-p_3.
\]
The set $K$ is a fundamental domain for the lattice $\Gamma=M\Z^2$, where 
\[
M=[w_1 |w_2], \qquad w_1=\left[\begin{matrix} 0\\ \sqrt{3} \end{matrix}\right], \quad w_2=\left[\begin{matrix} 3/2\\ \sqrt{3}/2 \end{matrix}\right].
\]
Then we transform $K$ so that $M^{-1}K$ is a fundamental domain for the lattice $\Z^2$, see Figure 1.

\begin{figure}
\includegraphics{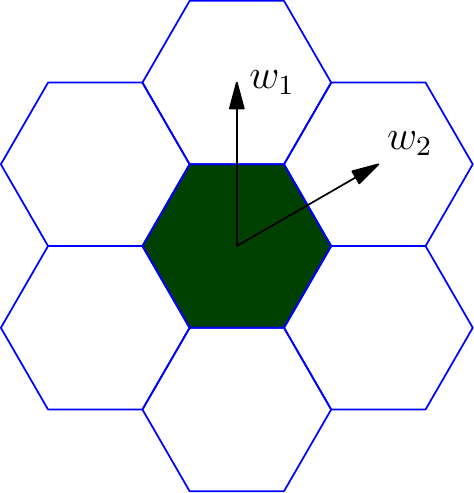}
\hskip1cm
\includegraphics{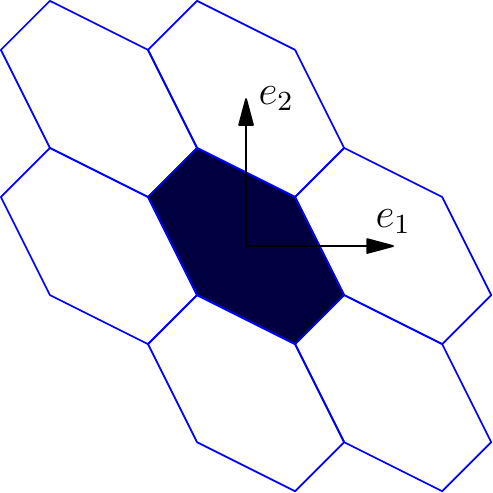}
\caption{Sets $K$ and  $M^{-1}K$.}
\end{figure}

Next we consider a grid $1/n \Z^2$, where $n$ is a scaling parameter. 
We color all cubes which have nonempty intersection with $M^{-1}K$.
 If a scaling parameter $n$ is sufficiently small we have all cubes in
$\epsilon$ neighborhood of $M^{-1}K$, see Figure 2.  To construct orthogonal projection from Theorem \ref{lat}
 we need to choose cubes that form a fundamental domain for the lattice $\Z^2$ by eliminating redundant cubes, see Figure 2.

\begin{figure}
\includegraphics{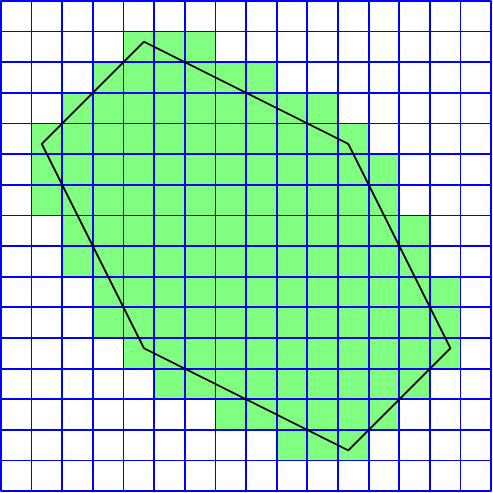}
\hskip1cm
\includegraphics{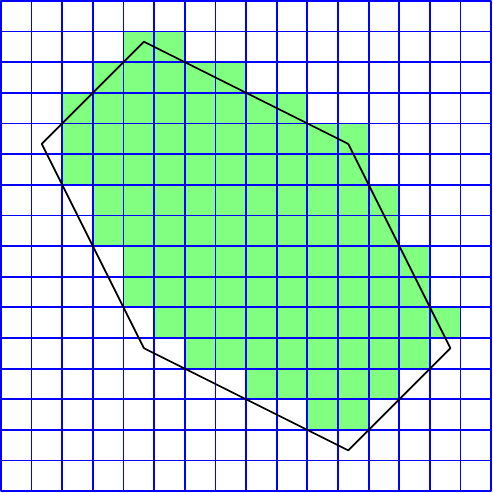}
\caption{Construction in Theorem \ref{lat} for scaling parameter $n=10$.} 
\end{figure}


\begin{corollary}
Let $B$ be a ball in the torus $\T^d=\R^d/\Z^d$. Then there exists a discrete subgroup $G \subset \T^d$ and a Hestenes operator $P$, which is orthogonal projection localized on $B$, such that 
\[
\sum_{\gamma \in G } T_\gamma P T_{-\gamma} f= f
\qquad\text{for all } f\in L^2(\T^d).
\]
In particular, projections 
$T_\gamma P T_{-\gamma}$, $\gamma \in \Gamma$, are mutually orthogonal.
\end{corollary}

\begin{proof}
Let $p:\R^d \to \T^d=\R^d/\Z^d$ be the quotient map. Then, a ball $B$ in the torus $\T^d=\R^d/\Z^d$ is of the form $B=p(\mathbf B(x,r))$, where $\mathbf B(x,r)$ is a ball in $\R^d$. Without loss of generality, we can assume that $r<1/(2\sqrt{d})$, so that the balls $\mathbf B(x+k,r)$, $k\in\Z^d$, are disjoint. Choose sufficiently large $n\in \N$ such $x+[0,1/n]^d \subset \mathbf B(x,r)$. Then, $K=x+[0,1/n]^d$ is a fundamental domain of $\R^d/\Gamma$, where $\Gamma=n^{-1}\Z^d$. By Theorem \ref{lat} there exists a Hestenes operator $P'$ on $\R^d$, which is localized in $\mathbf B(x,r)$, such that $P'$ is an orthogonal projection satisfying \eqref{lat5}. Define $\Z^d$-periodization of $P'$ by
\[
P= \sum_{k\in \Z^d} T_k P' T_{-k}.
\]
We can treat $P$ as a Hestenes operator on $\T^d$, which is an orthogonal projection on $L^2(\T^d)$ localized on $B$. This follows from the fact that $P'$ is localized in $\mathbf B(x,r)$ and the balls $\mathbf B(x+k,r)$, $k\in\Z^d$, are disjoint. Hence, we obtain the conclusion for the group $G=(n^{-1}\Z^d)/\Z^d$.
\end{proof}

We end this section with a continuous analogue of Theorem \ref{lat} on the real line, which motivates results in subsequent sections.

\begin{proposition}\label{prosta}
For fixed $\delta>0$ and $\alpha<\beta$ satisfying $\frac{\beta-\alpha}{2}>\delta$, let $P_{[\alpha,\beta]}$ be a smooth orthogonal projection given by \eqref{exp}.
For any continuous function $f:\R\to \R$ and any $t\in \R$, we have
\[
\int_\R T_\xi P_{[\alpha,\beta]} T_{-\xi} f(t) d\xi =
\int_\R P_{[\xi+\alpha,\xi+\beta]} f(t)d\xi=(\beta-\alpha)f(t).
\]
\end{proposition}

\begin{proof}
The first equality follows by \eqref{translacja}.
By \eqref{exp} we have
\[
\begin{aligned}
\int_\R & P_{[\xi+\alpha,\xi+\beta]} f(t) d\xi= \int_{t-\beta+\delta}^{t-\alpha-\delta} f(t)d\xi
\\
&+\int_{t-\alpha-\delta}^{t-\alpha+\delta}
s^2(t-(\alpha+\xi))f(t)+ s(t-(\alpha+\xi))s(\alpha+\xi-t)f(2(\alpha+\xi)-t)d\xi
\\
&+\int_{t-\beta-\delta}^{t-\beta+\delta}s^2(\beta+\xi-t) f(t) - s(t-(\beta+\xi))s(\beta+\xi-t)f(2(\beta+\xi)-t) d\xi.
\end{aligned}
\]
Since $P_{[\alpha,\beta]}$ has opposite polarities at endpoints, the change of variables yields
\[
\begin{aligned}
\int_\R  P_{[\xi+\alpha,\xi+\beta]} f(t) d\xi
& =f(t)(\beta-\alpha-2\delta)+2 f(t) \int_{-\delta}^\delta s^2(u) du
\\
& =f(t)(\beta-\alpha-2\delta)+2 f(t) \int_{0}^\delta (s^2(u)+s^2(-u)) du=(\beta-\alpha)f(t).
\end{aligned}
\]
The last equality follows from \eqref{sfun}.
\end{proof}

\section{Averages of smooth projections on $\Sp^1$}\label{S4}

In this section we show that
the Marcinkiewicz average of a smooth projection on an arc in $\Sp^1=\{z\in \C: |z|=1\}$ is a multiple of the identity.

\begin{definition}\label{pq}
Let $P$ be a Hestenes operator on $\R$, localized on $(a,b)$ with $b-a < 2 \pi$. 
Take $\rho$ such that $\rho < a < b < \rho + 2\pi$. 
Define an operator $\widetilde{P}$ acting on a function $f:\mathbb S^1\to \R$ by
$$
\widetilde{P}f(e^{it}) = P (f \circ \Psi_1)(t), \qquad t \in [\rho, \rho+ 2 \pi),
$$
where $\Psi_1(t)=e^{it}$.
Then $\widetilde{P} $ is a Hestenes operator on $\Sp^1$, localized on an arc $Q = \Psi_1((a,b)) \subset \Sp^1$. 
In particular, localization of $P$ on $(a,b)$ implies that $\widetilde{P}f (w) = 0$ for $w \in \Sp^1 \setminus Q$. 
This implies that  definition of $\widetilde {P}$
does not depend on $\rho$, provided $\rho < a < b < \rho + 2\pi$.
\end{definition}

Fix $\alpha<\beta$ and  $0<\delta<\frac{\beta-\alpha}{2}$. Define an operator $R_\alpha$ acting on functions $f$ on $\R$ by
\[
R_\alpha f(t) =  s(t-\alpha)s(\alpha-t)f(2\alpha-t) \qquad\text{for } t \in \R.
\]
Define a multiplication operator $Mf (t) = m(t)f(t)$, with
$$
m(t) = \begin{cases} 0 & \text{for } t < \alpha - \delta , 
\\
s^2(t-\alpha) & \text{for } t \in [\alpha-\delta, \alpha + \delta] ,
\\ 
1 & \text{for } t \in (\alpha + \delta, \beta - \delta) ,
\\
 s^2(\beta - t) & \text{for } t \in [\beta-\delta, \beta + \delta] ,
\\
 0 & \text{for } t > \beta + \delta . 
\end{cases}
$$
Then, the operator $P_{[\alpha, \beta]}$, given by formula \eqref{exp}, satisfies
\begin{equation}\label{mrr}
P_{[\alpha,\beta]} = M + R_\alpha  - R_\beta.
\end{equation}
Observe  $M$, $R_\alpha$, and $R_\beta$ are simple Hestenes operators localized on intervals $(\alpha- \delta, \beta +\delta)$, $(\alpha - \delta, \alpha +\delta)$, and $(\beta - \delta, \beta + \delta)$, respectively.
Note that
\begin{equation}\label{tr}
T_{\xi} R_\alpha T_{-\xi} f(t) =  s(t-(\alpha+\xi))s(\alpha+\xi-t)f(2(\alpha+\xi)-t) = R_{\alpha+\xi}f(t).
\end{equation}
Hence,
\begin{eqnarray*}
 T_\xi P_{[\alpha,\beta]} T_{-\xi}
 & = &
T_\xi M T_{-\xi} + R_{\alpha + \xi} - R_{\beta+ \xi},
\end{eqnarray*}
and $T_\xi M T_{-\xi} f(t) = m(t-\xi) f(t)$.

In the sequel, we need to consider both translation operators on $\R$ and on $\Sp^1$.
To distinguish between these two operators, we denote a translation (rotation) operator $\tau_z$ on $\Sp^1$ by $ \tau_z f(w) = f(z^{-1} w)$, where
$f:\Sp^1 \to \R$ and $z,w \in \Sp^1$.

\begin{lemma}\label{pxi}
Let $P$ be a Hestenes operator localized on an interval $(a,b)$ with $b-a < 2 \pi$. Define a Hestenes operator 
$\widetilde {P}$ on $\Sp^1$ by Definition \ref{pq}.
Then, $P_\xi = T_\xi P T_{-\xi} $ is a Hestenes operator localized on $(a+\xi, b+\xi)$ and $\widetilde{P_\xi}$ is defined as well. 
Moreover, we have
\begin{equation}\label{pxi0}
\widetilde{ P_\xi } = \tau_z \widetilde{P} \tau_{z^{-1}},
\quad \hbox{ where } \quad 
z = e^{i \xi} .
\end{equation}
\end{lemma}

\begin{proof}
The fact that $P_\xi$ is a Hestenes operator localized on $(a+\xi, b+\xi)$ follows from an explicit formula for $P_\xi$ when $P$ is a simple Hestenes operator.
To verify \eqref{pxi0}, take $f : \Sp^1 \to \R$. Observe first that for $u\in \R$,
$$
T_u (f \circ \Psi_1) = (\tau_{e^{iu}} f) \circ \Psi_1.
$$
Indeed, we have
\begin{align*}
T_u (f \circ \Psi_1) (t)  & =  (f \circ \Psi_1)(t-u) 
= f(e^{it} e^{-iu}) 
\\  & =  (\tau_{e^{iu} }f)(e^{it}) = (\tau_{e^{iu}} f) \circ \Psi_1 (t).
\end{align*}
Fix $\rho$ such that $\rho < a < b < \rho + 2 \pi$. Clearly, $\rho+ \xi < a + \xi <  b+\xi  < \rho +\xi+ 2 \pi$,
and $ t \in [\rho + \xi ,  \rho +  \xi + 2 \pi)$ if and only if $t - \xi \in [\rho  , \rho + 2 \pi)$
Therefore, for $ t \in [\rho + \xi ,  \rho +  \xi + 2 \pi)$
\begin{align*}
\widetilde{ P_\xi } f (e^{it}) & = 
P_\xi (f \circ \Psi_1)(t) = T_\xi P T_{-\xi} (f \circ \Psi_1)(t)
\\ & =  P( T_{-\xi}(f \circ \Psi_1))(t-\xi) 
= P( (\tau_{e^{-i\xi}} f )\circ \Psi_1)(t-\xi)
\\ & =   \widetilde{P}( \tau_{z^{-1}}f)(e^{i(t-\xi)}) 
=  \widetilde{P}( \tau_{z^{-1}}f)(e^{it} z^{-1})
= \tau_z  \widetilde{P}\tau_{z^{-1}}f (e^{it}). \qedhere
\end{align*}
\end{proof}

Since $SO(2)\approx \Sp^1$ with normalized Haar measure $\mu$, the Marcinkiewicz average of an operator $P$ is given by 
$$
{\mathcal S}(P)f(w) = \int_{\Sp^1} \tau_z P \tau_{z^{-1}} f (w) d\mu(z), \qquad w\in \Sp^1.
$$

\begin{theorem}\label{uu}
Let $\alpha<\beta$ be such that $\beta-\alpha< 2 \pi$. Let $\delta>0$ be such that 
\begin{equation}\label{dab}
2\delta< \min(\beta - \alpha, 2\pi -(\beta-\alpha)).
\end{equation}
For $Q = \Psi_1([\alpha,\beta])$, consider an operator $P_Q = \widetilde{P_{[\alpha,\beta]}}$ as in Definition \ref{pq}.
Then, for any continuous function $f:\Sp^1\to \R$ and any $w\in \Sp^1$, the Marcinkiewicz average satisfies
\begin{equation}\label{us}
\calS(P_{Q})(f)(w)=\frac{\beta-\alpha}{2\pi} f(w).
\end{equation}

\end{theorem}
\begin{proof}
Denote $\kappa = \beta - \alpha$ and $v = e^{i \kappa}$.
By \eqref{tr} 
$$
R_\beta = R_{\alpha + \kappa} = T_\kappa R_\alpha T_{-\kappa},
$$
and consequently by Lemma \ref{pxi} we have
\begin{equation}\label{us2}
\widetilde{R_\beta} = \tau_v \widetilde{R_\alpha}\tau_{v^{-1}}.
\end{equation}
Therefore,
\begin{equation}\label{us4}
\tau_z \widetilde{R_\beta} \tau_{z^{-1}} = \tau_z \tau_v \widetilde{R_\alpha}\tau_{v^{-1}} \tau_{z^{-1}}
= \tau_{zv} \widetilde{R_\alpha} \tau_{(zv)^{-1}}.
\end{equation}
Observe that $\widetilde{M}f=m_Qf$, where $m_Q$ is a function on $\mathbb S^1$ given by 
\[
m_Q (e^{it}) = \begin{cases}
s^2(t-\alpha) &
t \in [\alpha-\delta,\alpha+\delta],
\\
1 & t\in (\alpha+\delta,\beta-\delta),
\\
s^2(\beta-t)  &
t\in [\beta-\delta,2\pi+\alpha-\delta).
\end{cases}
\]
By \eqref{mrr} and \eqref{us2} we have
$$
P_Q = \widetilde{P_{[\alpha,\beta]}} 
= \widetilde{M}+ \widetilde{R_\alpha}  - \tau_v \widetilde{R_\alpha}\tau_{v^{-1}}.
$$
By \eqref{us4} this implies that
\[
\tau_z P_Q \tau_{z^{-1}}   =  \tau_z \widetilde{M} \tau_{z^{-1}} + 
\tau_z  \widetilde{R_\alpha} \tau_{z^{-1}} - \tau_{zv} \widetilde{R_\alpha} \tau_{(zv)^{-1}}.
\]
Further, note that 
$$
\tau_z \widetilde{M} \tau_{z^{-1}} f (w) = m_Q(wz^{-1}) f(w).
$$
Summarizing, we get
\begin{equation}
\tau_z P_Q \tau_{z^{-1}} f (w) =  m_Q(wz^{-1}) f(w) + 
\tau_z  \widetilde{R_\alpha} \tau_{z^{-1}} f(w) - \tau_{zv} \widetilde{R_\alpha} \tau_{(zv)^{-1}} f(w).
\label{aaa}
\end{equation}
By the invariance of Haar measure applied to $g(z) = \tau_z  \widetilde{R_\alpha} \tau_{z^{-1}} f(w)$ we see that
$$
\int_{\Sp^1} \tau_z  \widetilde{R_\alpha} \tau_{z^{-1}} f(w) d \mu(z)
= \int_{\Sp^1}  \tau_{zv} \widetilde{R_\alpha} \tau_{(zv)^{-1}} f(w)d \mu(z).
$$
Therefore,  integrating \eqref{aaa} over $\Sp^1$ we obtain 
\begin{eqnarray*}
\int_{\Sp^1} \tau_z P_Q \tau_{z^{-1}} f (w) d\mu(z)  =  f(w) \int_{\Sp^1} m_Q(wz^{-1}) d\mu(z) =  f(w)  \int_{\Sp^1} m_Q(z) d\mu(z).
\end{eqnarray*}
The conclusion follows from the fact that 
\[
\int_{\Sp^1} m_Q(z) d\mu(z)= \frac{1}{2\pi} \int_{\alpha-\delta}^{2\pi+\alpha-\delta} m(t) dt
= \frac{1}{2\pi} \int_{\R} m(t) dt
=
\frac{\beta - \alpha}{2 \pi}.
\qedhere
\]
\end{proof}

\section{Latitudinal  projections on sphere}\label{S5}

In this section we define latitudinal operators, whose action depends only on latitude variable, by transplanting one dimensional Hestenes operators to meridians. We also show that the Marcinkiewicz average of latitudinal projection is a multiple of the identity.

For $k\geq 2$, we define a surjective function 
\[
\Phi_{k}:[0,\pi]\times \mathbb S^{k-1} \to \mathbb S^{k}
\]
by the formula
\begin{equation}\label{Phi}
\Phi_{k}(\th,\xi)=(\xi\sin \th,\cos \th),
\qquad\text{where }(\th,\xi)\in [0,\pi]\times \mathbb S^{k-1}.
\end{equation}
Note that $\Phi_{k}$ is a diffeomorphism
\[
\Phi_{k}:(0,\pi)\times \mathbb S^{k-1} \to \mathbb S^{k}\setminus \{\1^{k},-\1^{k}\},
\]
where $
 \1^{k}=(0,\ldots,0,1)\in \mathbb S^{k}$ is the ``North Pole".
 Let $d_g$ be Riemannian metric on a sphere and let $\1=\1^{d-1}$.
 Note that for $\xi\in \Sp^{d-1}$, $d_g(\1,\xi)=t$, where $\lan \1,\xi \ran=\cos t$.

\begin{definition}
Let $P:C[0,\pi] \to C[0,\pi]$ be a continuous operator.  
 For fixed $k \geq 2$, let $\mathbf I$ be the  identity operator on $C(\Sp^{k-1})$. Define an operator 
\[
P\otimes \mathbf I:  C([0,\pi] \times \Sp^{k-1}) \to C([0,\pi] \times \Sp^{k-1}),
\] 
acting on a continuous function $g$ on $[0,\pi] \times \Sp^{k-1}$  by  
\[
(P\otimes \mathbf I) g  (t,y) = P\left(g( \cdot, y)\right)(t), \qquad (t,y)\in [0,\pi] \times \Sp^{k-1}.
\]
\end{definition}

It can be checked by direct calculations that if $P,Q: C[0,\pi] \to C[0,\pi]$, then
\begin{equation}\label{match}
(P\otimes \mathbf I)\circ (Q\otimes \mathbf I)= (P\circ Q)\otimes \mathbf I.
\end{equation}

\begin{definition}\label{krzyzyk}
Let 
\[
C_0([0,\pi])=\{f\in C([0,\pi]): f(0)=f(\pi)=0 \}.
\]
Let $P:C[0,\pi] \to C[0,\pi]$ be a continuous linear operator such that 
\begin{equation}\label{zero}
P(C_0[0,\pi]) \subset  C_0[0,\pi].
\end{equation}
We define a latitudinal  operator    acting on  $f\in C(\Sp^k)$ by
\[
P^\#f(\xi)  = \begin{cases}\left(  P\otimes \mathbf I(f \circ \Phi_{k}) \right) ( \Phi_{k}^{-1}(\xi)), & \xi \in \Sp^{k} \setminus \{\1^{k},-\1^{k}\} \\
P\otimes \mathbf I (f\circ \Phi_{k})(0,\1^{k-1}), & \xi=\1^k\\
P\otimes \mathbf I (f\circ \Phi_{k})(\pi,\1^{k-1}), & \xi=-\1^k.\\
\end{cases}
\]
\end{definition}

\begin{lemma}\label{oznaka}
If $P :C[0,\pi] \to C[0,\pi]$  satisfies condition  \eqref{zero}, then $P^\#:C(\Sp^k) \to C(\Sp^k)$. 
\end{lemma}

\begin{proof}
Denote
\[
C_b([0,\pi] \times \Sp^{k-1})=\{ g\in  C([0,\pi] \times \Sp^{k-1}) : \exists_{a_0, a_\pi} \forall_{\xi\in \Sp^{k-1}} g(0,\xi)=a_0, g(\pi,\xi)=a_\pi \}.
\]
Let $f$ be a function on $\Sp^k$.  Then $f\in C(\Sp^k)$ if and only if
$f\circ \Phi_k \in C_b([0,\pi] \times \Sp^{k-1})$. Indeed  $\Sp^k$ is homomorphic with the quotient space $[0,\pi] \times \Sp^{k-1}/\sim$,
which identifies $\{0\}\times\Sp^{k-1}$ and $\{\pi\}\times\Sp^{k-1}$ with single points corresponding to poles $\1^k$ and $-\1^k$, respectively.

The assumption  $P(C_0[0,\pi]) \subset  C_0[0,\pi]$ guarantees that
\[
P\otimes \mathbf I (C_b([0,\pi] \times \Sp^{k-1})) \subset C_b([0,\pi] \times \Sp^{k-1}).
\]
Indeed, let $g\in C_b([0,\pi] \times \Sp^{k-1})$ and $a_0=g(0,\1^{k-1})$ and $a_\pi=g(\pi,\1^{k-1})$. Define $p(t,y)=\frac{\pi-t}{\pi}$, $q(t,y)=\frac{t}{\pi}$ and
\[
h(t,y)=g(t,y)-a_0 p(t,y)-a_\pi q(t,y),\quad (t,y)\in [0,\pi]\times \Sp^{k-1}.
\]
Consequently for all $y\in \Sp^{k-1}$
\[
h(0,y)=h(\pi,y)=0.
\]
Hence 
\[
P\otimes  \mathbf I (h)(0,y)=P\otimes  \mathbf I (h)(\pi,y)=0.
\]
 We conclude that 
 \[
 P\otimes  \mathbf I(g)(0,y)=a_0 (P\otimes  \mathbf I)(p)(0,y) +a_\pi(P\otimes  \mathbf I)(q)(0,y).
 \]
Since  $p$ and $q$ do not depend on $y\in \Sp^{k-1}$,  functions 
$(P\otimes  \mathbf I)(p)(0,y)$ and $(P\otimes  \mathbf I)(q)(0,y)$ also do not depend on $y\in \Sp^{k-1}$.
Hence, $P\otimes  \mathbf I(g)$ is constant on $\{0\}\times \Sp^{k-1}$.
The same argument shows that $P\otimes  \mathbf I(g)$ is constant on $\{\pi\}\times \Sp^{k-1}$.
\end{proof}

\begin{lemma}\label{krzyzyki0}
If $P,Q :C[0,\pi] \to C[0,\pi]$ both satisfy condition  \eqref{zero},
then  
\begin{equation}\label{krzyzyki}
(P\circ Q)^\# = P^\# \circ Q^\#.
\end{equation}
\end{lemma}

\begin{proof}
By \eqref{match} and Definition \ref{krzyzyk} the formula \eqref{krzyzyki}  holds for continuous functions $f$ on $\Sp^k$ which vanish on poles. Let $p$ and $q$ be as in the proof of Lemma \ref{oznaka}. Likewise, \eqref{krzyzyki} holds for $p\circ \Phi_k^{-1}$ and $q\circ \Phi_k^{-1}$. Since any function $f$ on $\Sp^k$ is a linear combination of $p\circ \Phi_k^{-1}$, $q\circ \Phi_k^{-1}$, and a function vanishing on poles, the  formula \eqref{krzyzyki} holds for all $f\in C(\Sp^k)$.
\end{proof}

For further reference let $\rho: C([0,\pi]) \to C([0,\pi])$ be a  reflection operator given by  
\[
\rho f (t) = f(\pi -t)\qquad \text{for } f\in C([0,\pi]). 
\]
Let $R = \rho\otimes \mathbf I$, where $\mathbf I$ is the  identity operator on $C(\Sp^{k-1})$.
Then  
\[
Rg(t,y)=g(\pi-t,y)\qquad \text{for }  g\in C([0,\pi]\times \Sp^{k-1}).
\]
By Definition \ref{krzyzyk} we have 
\[
\rho^\# f (\xi)= f(\xi_1,\ldots,\xi_k,-\xi_{k+1})\qquad \text{for } \xi=(\xi_1,\ldots,\xi_{k+1})\in \Sp^k, f\in C(\Sp^{k}).
\]

\begin{lemma}\label{l55}
  Fix  $k \geq 2$.  Let $L :C[0,\pi] \to C[0,\pi]$ be a continuous operator and  $\eta\in SO(k+1)$. Then,  
\begin{equation}\label{rho2}
T_\eta L^\#  T_{\eta^{-1}} = \begin{cases}  L^\# & \text{if } \eta(\1)=\1,
\\
\rho^\# L^\# \rho^\# & \text{if } \eta(\1)=-\1.
\end{cases}
\end{equation}
\end{lemma}

\begin{proof}
Suppose that $\eta(\1)=-\1$. Then $\eta$ is a block diagonal matrix with two blocks:  $C\in O(k)$ and $-1$ in the last diagonal entry. 
Hence, for parametrization  $\xi=\Phi_{k}(t,y)$ of sphere $\Sp^k$,
we have  $\eta(\xi)=\Phi_{k}(\pi-t,C y)$ for a certain matrix $C\in O(k)$.
Consequently
\begin{equation*}\label{iff12}
\eta^{-1} (\xi)=\Phi_{k}(\pi-t,C^{-1}y).
\end{equation*}
Take $f\in C(\Sp^k)$. Letting  $g=f\circ \Phi_{k}$,  we have
\begin{equation*}\label{iff14}
T_\eta f(\xi)=f(\eta^{-1}\xi)=g(\pi-t,C^{-1}y).
\end{equation*}
Let  $g_\eta=T_{\eta^{-1}} f\circ \Phi_{k}$. Then we have
\[
\begin{aligned}
T_\eta \left(L^\#  T_{\eta^{-1}}f \right)(\xi) &= L^\#  T_{\eta^{-1}}(f)\Phi_{k}(\pi-t,C^{-1}y)
\\
&=L\otimes \mathbf I  (T_{\eta^{-1}}f \circ \Phi_{k})(\pi-t,C^{-1}y) 
=R(L\otimes \mathbf I)(g_\eta)(t,C^{-1}y).
\end{aligned}
\]
Since $g_\eta(t',y')=R g(t',C y')$ and operators $R$ and $L\otimes \mathbf I$ act only on the first variable $t$, we have
\[
(L\otimes \mathbf I)(g_\eta)(t,C^{-1}y)=
L(g_\eta(\cdot,C^{-1}y))(t)
=L(Rg(\cdot,y))(t) = (L\otimes \mathbf I)R g(t,y).
\]
Therefore, $R = \rho\otimes \mathbf I$ yields
\begin{equation*}\label{dolematu}
T_\eta L^\#  T_{\eta^{-1}}f (\xi)= R(L\otimes \mathbf I)(R g)(t,y) = (\rho L\rho \otimes \mathbf I)g(t,y).
\end{equation*}
Hence, by Definition \ref{krzyzyk} and Lemma \ref{krzyzyki0} 
\[
T_\eta L^\#  T_{\eta^{-1}}f (\xi) = (\rho L\rho)^\# f(\xi) = \rho^\# L^\# \rho^\# f(\xi).
\]
In the case $\eta(\1)=\1$, the proof follows similar arguments using a representation  $\eta(\xi)=\Phi_{k}(t,C y)$ for a certain matrix $C\in SO(k)$.
\end{proof}

\begin{corollary}\label{symetria} Fix  $k \geq 2$. Let $L :C[0,\pi] \to C[0,\pi]$ be a continuous operator which satisfies condition \eqref{zero}.
Let $K = L - \rho L \rho$.   Then for $f\in C(\Sp^k)$ and $\xi\in \Sp^k$,
\[
{\calS}(K^\#) f(\xi)= \int_{SO(k+1)} T_b K^\# T_{b^{-1}}f(\xi) d\mu_{k+1}(b)=0.
\]
\end{corollary}

\begin{proof}
Take any   $\eta\in SO(k+1)$  such that  $\eta(\1)=-\1$. By Lemma \ref{l55} we have
\begin{equation}\label{rho}
\rho^\# L^\# \rho^\#=T_\eta L^\#  T_{\eta^{-1}}
\end{equation}
Then, the invariance of measure $\mu_{k+1}$ yields
\[
\calS(\rho^\# L^\# \rho^\#)=\calS(T_\eta  L^\# T_{\eta^{-1}})=\calS( L^\#). \qedhere
\]
\end{proof}


 Let $\vartheta, \delta $ be such that $0< \vartheta - \delta < \vartheta + \delta < \pi$. Define
$$
L_\vartheta f(t) = s(t-\vartheta)s(\vartheta - t) \left( \frac{ \sin(2 \vartheta - t )}{\sin t} \right)^{(k-1)/2}
f(2\vartheta -t).
$$
It can be checked by a direct calculation that
$$
L_{\pi - \vartheta} = \rho L_\vartheta \rho.
$$
Next, for $0 < \vartheta < \pi/2$ and suitable $\delta>0$, define function $\psi_\vartheta$ by formula
\[
\psi_\vartheta(t) = \begin{cases} 0 &   t < \vartheta - \delta , 
\cr
s^2(t-\vartheta) &   t \in [\vartheta-\delta, \vartheta + \delta] ,
\cr 1 & t \in (\vartheta  + \delta, \pi - \vartheta - \delta) ,
\cr s^2(\pi - \vartheta - t) & t \in [\pi - \vartheta-\delta, \pi -\vartheta + \delta] ,
\cr 0 &  t > \pi - \vartheta + \delta. 
\cr
\end{cases}
\]
Define
$$
P_\vartheta   = M_{\psi_\vartheta} + L_\vartheta - L_{\pi - \vartheta} = M_{\psi_\vartheta} + L_\vartheta - \rho L_{\vartheta} \rho,
$$
where $M_{\psi_\vartheta} (f)=\psi_\vartheta f$ denotes the multiplication operator.

Next, observe that there is a function $\psi^\#_\vartheta \in C^\infty(\Sp^{k})$ such that
$$
( M_{\psi_\vartheta} )^\# = M_{\psi^\#_\vartheta}.
$$
Let
$$K_\vartheta = L_\vartheta - L_{\pi - \vartheta} =  L_\vartheta - \rho L_{\vartheta} \rho.$$
Define and operator $U: C(\Sp^k) \to C(\Sp^k)$ by
$$
U= P_\vartheta ^\# = ( M_{\psi_\vartheta} )^\# + K_\vartheta^\# = M_{\psi^\#_\vartheta} + K_\vartheta^\#.
$$

\begin{theorem}\label{operatorU}
Fix  $k \geq 2$.
Let $\vartheta, \delta $ be such that $0< \vartheta - \delta < \vartheta + \delta < \pi/2$.
Then, $U$ is a Hestenes operator localized on the latitudinal strip $\Phi_k((\th-\delta, \pi - \th +\delta) \times \mathbb S^{k-1})$, $U$ extends to an orthogonal projection on $L^2(\Sp^k)$, and  
\begin{equation}\label{sup}
\calS (U)f(\xi)=C(\psi^\#_\vartheta) f(\xi)
\qquad\text{for all } f\in C(\Sp^k), \ \xi \in \Sp^k.
\end{equation}
\end{theorem}

\begin{proof}
Let  $E_{\th}$ be an AWW operator from 
 \cite[Definition 3.4]{BD}, see also \cite[(3.5),(3.6)]{BD}. That is, for $g: [0,\pi] \to \C$ we define
\begin{equation}\label{1.7a}
E_{\th}(g)(t)=
\begin{cases}
g(\th)& t>\th +\delta,\\
0& t<\th-\delta.\\
\end{cases}
\end{equation} 
For $t\in[\th-\delta,\th+\delta]$ we define
  \begin{equation}\label{operatorE}
\begin{aligned}
E_{\th}(g)(t)
=&  s^2(t-\th)g(t)
\\&+ s(t-\th)s(\th-t)\left(\frac{\sin(2\th-t)}{\sin t}\right)^{(k-1)/2} g(2\th-t).
\end{aligned}
\end{equation}
The above formula also holds for $t$ outside of $[\th-\delta,\th+\delta]$, since $s(t-\th)s(\th-t)= 0$  and we can ignore the second term in \eqref{operatorE}. 

By \cite[ Lemma 3.3]{BD} the operator $(E_\th)^\# \in \mathcal H(\Sp^k)$ and $(E_\th)^\#$ extends to an orthogonal projection on $L^2(\Sp^k)$. Since 
$P_\th= E_{\th}-E_{\pi-\th}$, we have
\[
U=E_{\th}^\#-E_{\pi-\th}^\#.
\]
The fact that $U$ is an orthogonal projection follows from \cite[ Lemma 3.4]{BD}.
By Lemma \ref{perturbacja} and  Corollary \ref{symetria} we deduce \eqref{sup}.
\end{proof}

\section{Averages of smooth orthogonal projections on sphere}\label{S6}

In this section we complete a construction of a smooth orthogonal projection, which is localized on arbitrarily small ball, such that its average is a multiple of the identity operator. To achieve this we will use the lifting procedure \cite[Definition 4.1]{BD}.

\begin{definition}
For $k\geq 2$, let
\[
C_{0}(\Sp^k) = \{ f \in C( \Sp^{k}): f(\1^{k}) = 0 = f(- \1^{k})\}.
\]
Suppose that $T: C(\Sp^{k-1}) \to C(\Sp^{k-1})$.  We define the lifted operator
$\hat{T}: C_0(\Sp^{k}) \to  C_0(\Sp^{k}) $ 
using the relation
\begin{equation}
\label{1.10}
\hat{T}(f)(t,\xi)=
\begin{cases}
T(f^t)(\xi)& (t,\xi)\in (0,\pi)\times \mathbb S^{k-1},
\\
0 & t=0\quad \text{or}  \quad t=\pi.
\end{cases}
\end{equation}
where
\[
f^t(\xi)=f(t,\xi), \quad (t,\xi)\in (0,\pi)\times \mathbb S^{k-1} \approx \mathbb S^{k}\setminus \{\1^{k},-\1^{k}\}.
\]
\end{definition}
It is easy to verify from \eqref{1.10} that if $f\in C_0(\Sp^{k})$, then $Tf\in C_0(\Sp^{k})$.
Moreover, the operator norms of $T$ and $\hat{T}$ are the same.

For $P:C(\Sp^k) \to C(\Sp^k)$ denote
\[
\calS_{k} (P) f(\zeta) = \int_{SO(k+1)}  T_b\circ P \circ T_{b^{-1}}f(\zeta)  \, d\mu_{k+1}(b)\qquad \text{for } f\in C(\Sp^k), \ \zeta\in \Sp^k.
\]

\begin{lemma}\label{zlematu.b}
Let $k\geq 2$. Let $P:C(\Sp^{k-1}) \to C(\Sp^{k-1}) $ be a continuous linear operator
such that
\begin{equation}\label{eq.101}
\calS_{k-1} (P) h  = c(P) h\qquad \text{for } h \in C(\Sp^{k-1}).
\end{equation}
Let $L: C[0,\pi] \to C_0[0,\pi]$ be a continuous linear operator.
Then the composition operator
$$
 \hat{P} \circ L^{\#}: C(\Sp^{k}) \to C(\Sp^{k}),
$$
satisfies 
\begin{equation}
\label{eq.108}
\calS_{k} (\hat{P} \circ L^{\#}) f= c(P) \calS_{k} (L^{\#}) f\qquad \text{for  } f \in C(\Sp^{k}).
\end{equation}

\end{lemma}

\begin{proof} 
Let $G=SO(k+1)$, $H=\{b\in SO(k+1): b(\1)=\1 \}\subset SO(k+1)$. We can identify $G/H=\Sp^k$.
For $x \in \Sp^{k}\setminus \{\1^k,-\1^k\}$, let $b_x \in SO(k+1)$ be  a rotation in the plane spanned by
$\{\1,x\}$ such that  $b_x(\1)=x$. 
Note that $b_\cdot$ is a continuous selector of coset representatives of $G/H$, 
\[
x\in \Sp^{k}\setminus \{\1^k,-\1^k\} \to b_x\in SO(k+1).
\]
Let $\sigma_k$ be a normalized Lebesgue measure on $\Sp^k$. 
By Weyl's formula \cite[Theorem 2.51]{Folland} for any $F\in C(G)$, we have
\begin{equation}\label{weyl}
\int_{SO(k+1)} F d\mu_{k+1} = \int_{\Sp^k} \int_{H} F(b_x a)  d\mu_{k}(a) d\sigma_{k}(x),
\end{equation}
where $\mu_k$ is a normalized Haar measure on $SO(k)$, which can be identified with $H$. That is, any $a\in H$ is a block diagonal matrix with two blocks:  $a' \in SO(k)$ and $1$ in the last diagonal entry.  
 
We claim that for  $h\in C_0(\Sp^{k})$ we have
\begin{equation}\label{stab}
\int_{H}    \left( T_a \circ \hat{P} \circ T_{a^{-1}} \right)h(\zeta) d\mu_k(a)=c(P) h(\zeta),\qquad \zeta\in \Sp^k.
\end{equation}
Since $T_a( C_0(\Sp^k))  \subset C_0(\Sp^k)$ for all $a\in H$, the formula \eqref{stab} holds trivially for $\zeta= \1^k,-\1^k$. Otherwise, any $\zeta \in \Sp^{k}\setminus \{\1^k,-\1^k\}$ can be identified with $(t,\xi) \in (0,\pi) \times \Sp^{k-1}$ through diffeomorphism $\Phi_k$. Hence, for any $(t,\xi) \in (0,\pi) \times \Sp^{k-1}$ we have
\[
\begin{aligned}
\int_{H}    \left( T_a \circ \hat{P} \circ T_{a^{-1}} \right)h(t,\xi) d\mu_k(a) & = \int_H P ((T_{a^{-1}} h)^t)((a')^{-1}\xi) d\mu_k(a)
\\
&= \int_{SO(k)} P (T_{(a')^{-1}} h^t)((a')^{-1}\xi) d\mu_k(a') 
\\
&= \int_{SO(k)} \left( T_{a'} \circ P \circ T_{(a')^{-1}} \right) h^t(\xi) d\mu_k(a')  = c(P) h(t,\xi).
\end{aligned}
\]
The last equality is a consequence of the assumption \eqref{eq.101}. Hence, \eqref{stab} holds.

Let $f\in C(\Sp^k)$  and $\zeta\in\Sp^{k}$. By \eqref{weyl} we have
\begin{equation*}\label{eq.103}
\begin{aligned}
\calS_{k}(\hat{P} \circ L^{\#})f(\zeta)
&=\int_{\Sp^k} \int_{H} T_{b_x a}\circ \hat{P} \circ L^{\#} \circ T_{(b_x a)^{-1}} f(\zeta)  d\mu_{k}(a) d\sigma_{k}(x)
\\
&=\int_{\Sp^{k}} \int_{H} T_{b_x}T_{a}\circ \hat{P} \circ T_{a^{-1}} T_{a} \circ L^{\#} \circ T_{a^{-1}}T_{b_x^{-1}}f(\zeta)  d\mu_{k}(a) d\sigma_{k}(x).
\end{aligned}
\end{equation*}
By \eqref{rho2} the above equals
\begin{align*}
& \int_{\Sp^{k}} \int_{H}  T_{b_x}  \left( T_a \circ \hat{P}\circ T_{a^{-1}} \right) L^\# T_{b_x^{-1}} f(\zeta)  d\mu_{k}(a) d\sigma_{k}(x)\\
&=  \int_{\Sp^{k}} \int_{H}    \left( T_a \circ \hat{P} \circ T_{a^{-1}} \right) L^\# T_{b_x^{-1}} f(b_x^{-1}\zeta)  d\mu_{k}(a) d\sigma_{k}(x).
\end{align*}
Hence, by \eqref{stab} 
\begin{align*}
 \calS_{k}(\hat{P} \circ L^{\#})f(\zeta) = & c(P) \int_{\Sp^{k}}     (   L^\# T_{b_x^{-1}} f) (b_x^{-1}\zeta)  d\sigma_{k}(x)
\\
= &
c(P)  \int_{\Sp^{k}}  T_{b_x} \circ    L^\#  \circ T_{b_x^{-1}} f (\zeta)  d\sigma_{k}(x).
\end{align*}
Applying again \eqref{rho2} and \eqref{weyl} yields
\[
 \calS_{k}(\hat{P} \circ L^{\#})f(\zeta) =
c(P)  \int_{\Sp^{k}}  \int_H T_{b_x}  T_a \circ    L^\#  \circ T_{a^{-1}} T_{b_x^{-1}} f (\zeta)  d\sigma_{k}(x)
=
c(P) \calS_k(L^\#)f(\zeta).
\qedhere 
\] 
\end{proof}

By the lifting lemma on the sphere \cite[Lemma 4.1]{BD}, or its generalization on Riemannian manifolds \cite[Lemma 5.3]{BDK}, we have the  following result.

\begin{lemma}
\label{zlematu.v}
Let $k\geq 2$.
Let $\vartheta, \delta $ be such that $0< \vartheta - \delta < \vartheta + \delta < \pi/2$.
 Let $U$ be a latitudinal orthogonal projection as in Theorem \ref{operatorU}.
Let   $P_{Q}$ be a Hestens operator on $\Sp^{k-1}$, which is localized on an open subset 
$Q \subset \Sp^{k-1}$, such that it induces an orthogonal projection on $L^2(\Sp^{k-1})$.
Then 
\[
P_{\Omega}=\hat{P_{Q}} \circ U=  U  \circ \hat{P_{Q}}
\]
is a Hestens operator on $\Sp^k$, which is localized on $\Omega=\Phi_k( (\vartheta-\delta, \pi-\vartheta+\delta)\times Q )$, and it induces an orthogonal projection on $L^2(\Sp^k)$.
\end{lemma}

Let 
\[
\Psi_k :[0,\pi]^{k-1}\times [0,2\pi] \to \mathbb S^{k}
\]
be the standard spherical coordinates given by the recurrence formula
\begin{equation}\label{scor}
\begin{aligned}
\Psi_1(t)&=(\sin t, \cos t), \qquad t\in [0,2\pi],
\\
\Psi_{k+1}(t,x)&=(\xi \sin t ,\cos t), \qquad
(t,x) \in [0,\pi] \times ([0,\pi]^{k-1}\times [0,2\pi]),
\end{aligned}
\end{equation}
where $\Psi_{k}(x)=\xi\in \mathbb S^{k}$.

To construct a Hestenes operator satisfying Theorem \ref{easyth} we will use two {\it symmetric interior patches}
\[
\begin{aligned}
Q &= \Psi_{k-1}( [\th^{k-1}_{1},\th^{k-1}_{2}]\times \cdots \times [\th^{2}_{1},\th^{2}_{2}]\times [\th^{1}_{1},\th^{1}_{2}]),
\\
\Omega &=\Psi_{k}( [\th^{k}_{1},\th^{k}_{2}]\times \cdots \times [\th^{2}_{1},\th^{2}_{2}]\times [\th^{1}_{1},\th^{1}_{2}] ),
\end{aligned}
\]
where $0<\th_1^j<\th_2^j<2\pi$ for $j=1$, and $0<\th_1^j<\th_2^j<\pi$, $\th_2^j=\pi-\th_1^j$ for $j=2,\ldots,k$. For sufficiently small $\delta>0$ define $\delta$-neighborhoods  of $\Omega$ and $Q$ by
\[
\begin{aligned}
Q_\delta &= \Psi_{k-1}( [\th^{k-1}_{1}-\delta,\th^{k-1}_{2}+\delta]\times \cdots \times [\th^{2}_{1}-\delta,\th^{2}_{2}+\delta]\times [\th^{1}_{1}-\delta,\th^{1}_{2}+\delta]),
\\
\Omega_\delta &=\Psi_{k}( [\th^{k}_{1}-\delta,\th^{k}_{2}+\delta]\times \cdots \times [\th^{2}_{1}-\delta,\th^{2}_{2}+\delta]\times [\th^{1}_{1}-\delta,\th^{1}_{2}+\delta]).
\end{aligned}
\]

\begin{theorem}\label{mainth}
Let $\Omega$ be a symmetric interior patch in $\Sp^k$, $k\ge 2$. Then there exist $\delta>0$ and  Hestenes operator  $P_{\Omega_\delta}$, which is an orthogonal projection localized on $\Omega_\delta$, such that 
for all $f\in C(\Sp^{k})$,
\begin{equation}\label{mainth2}
\int_{SO(k+1)}  T_b\circ P_{\Omega_\delta}\circ T_{b^{-1}} (f) d\mu_{k+1}(b) =c(P_{\Omega_\delta}) f,
\end{equation}
where $c(P_{\Omega_\delta})$ is a constant depending on $P_{\Omega_\delta}$. 
\end{theorem}

\begin{proof}
Let
\[
\Omega= \Psi_{k}( [\th^{k}_{1},\th^{k}_{2}]\times \cdots \times [\th^{2}_{1},\th^{2}_{2}]\times [\th^{1}_{1},\th^{1}_{2}])
\]
be a symmetric interior patch. 
For $j=2,\ldots,k$, let $U^j=E_{\th_1^j}^\#-E_{\th_2^j}^\#$  be the latitudinal projection corresponding to the interval $[\th^{j}_{1},\th^{j}_{2}]$ and acting on the space $C(\Sp^j)$ as in Theorem \ref{operatorU}.  

Suppose that $k=2$.
Let $Q=\{\Psi_1(t)=e^{it}\in \C: t\in [\th^{1}_{1},\th^{1}_{2}]\}$ be an arc in $\Sp^1\subset \C$. We choose $\delta>0$ such that 
$2\delta<\th^1_2 - \th^1_{1}$, $2\delta<2\pi -(\th^1_2-\th^1_{1})$, and $\th^2_1-\delta>0$.
 Then  by  Theorem \ref{uu} the operator $P_Q$ satisfies assumption \eqref{eq.101} of Lemma \ref{zlematu.b} with constant $c(P_Q)=\tfrac{\th^1_2-\th^1_{1}}{2\pi}$.
Applying   Lemma \ref{zlematu.b} and Lemma \ref{zlematu.v}  the operator
\[
P_{(2)}=U^2\circ \hat{P_Q}
\]  satisfies conditions of Theorem  \ref{mainth} with constant $c(P_{(2)})=c(U^2)c(P_Q)$.

For $k\geq 3$, we can assume by induction that we have an operator $P_{(k-1)}$ satisfying conclusions of Theorem \ref{mainth}.  Applying   Lemma \ref{zlematu.b} and Lemma \ref{zlematu.v}  for sufficiently small $\delta>0$,  the operator
\[
P_{\Omega_\delta}=P_{(k)}=U^k\circ \hat{P}_{(k-1)}
\]  satisfies conclusions of Theorem  \ref{mainth} with constant $c(P_{(k)})=c(U^k)c(P_{(k-1)})$.
 \end{proof}

We finish this section by showing a preliminary variant of Theorem \ref{easyth}.

\begin{theorem}\label{e1}
Let $\calB$ be a ball in $\Sp^{d-1}$. Let $\mu$ be a normalized Haar measure on $SO(d)$. 
There exist  Hestenes operator $P_\calB$  localized on $\calB$ and a constant $c=c(P_\calB)$ such that $P_\calB:L^2(\Sp^{d-1}) \to L^2(\Sp^{d-1})$ is an orthogonal projection and for all $f\in C(\Sp^{d-1})$,
\begin{equation}\label{e2}
\int_{SO(d)} T_b\circ P_\calB\circ T_{b^{-1}} (f) d\mu(b) =c f.
\end{equation}
\end{theorem}

\begin{proof}
Take any geodesic ball $\calB$ with radius $r>0$. For $\ve>0$, we choose $\th^j_1<\th^j_2$ and $\delta>0$, such that $\th^j_2-\th^j_1+2\delta<\ve$ for all $ j=1,\ldots,k$. Choose $\ve>0$ small enough such that symmetric interior patch $\Omega_\delta$ has diameter less than $r$. Let $a\in SO(d)$ be such that $a(\Omega_\delta) \subset \calB$. 
Define $P_\calB=T_a P_{\Omega_\delta} T_{a^{-1}}$, where $P_{\Omega_\delta}$ is as in Theorem \ref{mainth}. Then $P_\calB$ is both Hestenes operator and   orthogonal projection and moreover $P_\calB$ is localized in $\calB$. Indeed,
for any $f\in C(\Sp^{d-1})$, $\supp P_\calB f =a(\supp (P_{\Omega_\delta}\circ T_{a^{-1}} f)) \subset a(\Omega_\delta) $. Likewise, if $\supp f \cap \calB = \emptyset$, then $P_\calB f=0$. Hence, the localization of $P_\calB$ follows from \cite[Lemma 2.1]{BDK}.
Since $P_{\Omega_\delta}$ satisfies \eqref{e2} for $f\in C(\Sp^{d-1})$, so does $P_\calB$.
\end{proof}

\section{Proof of Theorem \ref{easyth}} \label{S7}

In this section we give a proof of 
Theorem \ref{easyth}, which is a consequence of Theorem \ref{e1} and the following two propositions. Let $\mathcal D$ be the test space of $C^\infty$ functions on $\Sp^{d-1}$. Let $\mathcal D'$ be the dual space of distributions on $\Sp^{d-1}$.

\begin{proposition}\label{apropo}
Let $P$ be Hestenes operator such that there is a constant $c=c(P)$ such that for all $f\in C(\Sp^{d-1})$ and for all $\xi\in \Sp^{d-1}$ the following reproducing formula holds
\begin{equation}\label{e3}
\int_{SO(d)} T_b\circ P \circ T_{b^{-1}} (f)(\xi) d\mu(b) =c f(\xi).
\end{equation}
Let $X$ be a quasi Banach space on which $X'$ separates points such that:
\begin{enumerate}
\item we have continuous embeddings ${\mathcal D} \hookrightarrow X \hookrightarrow{\mathcal D}'$ and ${\mathcal D}$ dense in $X$, 
\item there is a constant $C>0$ such that for all $b\in SO(d)$ 
\[
\|T_b\|_{X\to X}\leq C,
\]
\item the operator $P:X\to X$ is  bounded.
\end{enumerate}

Then the integral reproducing formula
\begin{equation}\label{e4}
\int_{SO(d)} T_b\circ P \circ T_{b^{-1}} (f) d\mu(b) =c f.
\end{equation}
 holds for all $f\in X$ in the sense of Pettis integral. In the case $X$ is Banach space the integral is Bochner integral. 
\end{proposition}

\begin{proof}
Observe that  the mapping $SO(d) \times \mathcal D\ni (b,f) \mapsto  T_b f\in \mathcal D$ is continuous. This follows from
\begin{equation}\label{e5}
|\nabla^k T_b f(x)| = |\nabla^k f(b^{-1}x)| \qquad \text{where } x\in \Sp^{d-1},\ b \in SO(d),\ k\ge 0,
\end{equation}
which can be seen from explicit formulas for covariant derivative $\nabla$ on the sphere \cite[(1.4.6) and (1.4.7)]{DaiXu}.

 By \cite[Lemma 3.2]{BD} or \cite[Theorem 2.6]{BDK}, the operator $P: \mathcal D \to \mathcal D$ is continuous.
By an argument as in the proof of \cite[Theorem 5.18]{Rudin}, the Pettis integral on the left hand side of \eqref{e4} exists and defines a continuous operator in the Fr\'echet space $\mathcal D$. By the assumption \eqref{e3}, this operator is a multiple of the identity operator by a constant $c=c(P)$. Hence, \eqref{e4} holds for $f\in\mathcal D$.

Note that conditions (1) and (2) imply that
\begin{equation}\label{e11}
SO(d)\times X \ni (b,f) \mapsto  T_b f\in X \text{ is continuous.}
\end{equation}
Since $\mathcal D \subset X$ is dense, for any $f_0\in X$ and $\ve>0$, there exists $g\in \mathcal D$ such that $||f_0-g||_X<\ve$. Since ${\mathcal D} \hookrightarrow X$ is a continuous embedding, for sufficiently close $b_1,b_2 \in SO(d)$, we have $\|T_{b_1} g- T_{b_2}g \|_X <\ve$.
By the triangle inequality for a quasi Banach space there exists a constant $K\ge 1$ such that
\[\begin{aligned}
\|T_{b_1} f_0- T_{b_2}f_0 \|_X
&\leq K (\|T_{b_1} f_0- T_{b_1}g \|_X + K( \|T_{b_1} g- T_{b_2}g \|_X + \|T_{b_2} g- T_{b_2}f_0 \|_X))  
\\
& \le K(C\ve + K(C+1)\ve).
\end{aligned}
\]
In the last step we used the assumption that operators $T_b$ are uniformly bounded.
On other hand, for any $f\in X$ such  that $||f-f_0|| < \ve$ we have
\[
\|T_{b_1} f_0- T_{b_2}f \|_X
\leq K (\|T_{b_1} f_0- T_{b_2}f_0 \|_X + \|T_{b_2} f_0- T_{b_2}f \|_X).
\]
Combing the above estimates yields \eqref{e11}.

Take any $f\in X$ and $\Lambda \in X'$. Then, the function
\begin{equation}\label{e13}
SO(d) \ni b \mapsto  \Lambda T_b P T_{b^{-1}} f\in X \text{ is continuous.}
\end{equation}
Hence, we can define a linear functional
\[
\Gamma(f):=\int_{SO(d)} \Lambda T_b P T_{b^{-1}} f \, d\mu_d(b).
\]
Moreover,
\[
|\Gamma(f)| \le \int_{SO(d)} |\Lambda T_b P T_{b^{-1}} (f)| \, d\mu_d(b)\le C^2 ||\Lambda|| \cdot ||P||_{X\to X} ||f||_X.
\]
Thus, $\Gamma \in X'$. Since $\Gamma(f)=c\Lambda(f)$ holds for $f\in \mathcal D$, it follows that the same holds for $f\in X$, and the conclusion follows by the definition of Pettis integral. Finally, if $X$ is a Banach space, then the integrand in \eqref{e4} is continuous, and hence, the integral exists in the Bochner sense.
\end{proof}

\begin{proposition}\label{fs}
The following spaces satisfy conditions (1)--(3) of Proposition \ref{apropo}:
\begin{itemize}
\item
Triebel-Lizorkin space $\F^s_{p,q}(\Sp^{d-1})$, $0< p<\infty, 0< q<\infty$, $s\in \R$,
\item
Besov space $\B^s_{p,q}(\Sp^{d-1})$, $ 0< p< \infty, 0<q< \infty$, $s\in\R$, 
\item
 the Lebesgue space $L^p(\Sp^{d-1})$ and Sobolev space $W^k_p(\Sp^{d-1})$, $1\le p<\infty$, $k\geq 1$,
\item
 the space $C^k(\Sp^{d-1})$, $k\geq 0$.
\end{itemize}
\end{proposition}

\begin{proof} The condition {\it (1)} is a standard fact in function spaces, whereas {\it (3)} follows from \cite[Theorem 2.6]{BDK} and \cite[Theorem 3.1 and Corollary 3.6]{BDK2}. The condition {\it (2)} is immediate for the spaces $L^p$, $W^k_p$, and $C^k$ from \eqref{e5}. The condition {\it (2)} is a consequence of a general result on smooth atomic decomposition for $\F^s_{p,q}$ and $\B^s_{p,q}$ spaces due to Skrzypczak \cite{Sk0}. Indeed, if $a$ is a smooth $(s,p)$-atom on $\Sp^{d-1}$ centered in $B(x,r)$, then its rotation $T_ba$ is also a smooth atom centered in $B(b^{-1}x,r)$, see 
\cite[Definition 6]{Sk0}. Hence, the atomic decomposition of $f\in \F^s_{p,q}$ (or $f \in \B^s_{p,q}$) of the form $f=\sum_{j=0}^\infty\sum_{i=0}^\infty s_{j,i} a_{j,i}$ as in \cite[Theorem 3]{Sk0} yields the atomic decomposition $T_bf=\sum_{j=0}^\infty\sum_{i=0}^\infty s_{j,i} T_b a_{j,i}$. While the centers of the family of atoms $\{a_{j,i}\}$ have changed after the rotation, they correspond to another uniformly finite sequence of coverings of $\Sp^{d-1}$ with the same parameters. Then, the equivalence of the norm $||f||_{\F^s_{p,q}}$ (or $||f||_{\B^s_{p,q}}$)  with its atomic decomposition norm is independent of the choice of such uniformly finite sequence of coverings. This can be seen by analyzing the proof of \cite[Theorem 3]{Sk0} to see that equivalence constants depends only on the parameters of a uniformly finite sequence of coverings. Alternatively, any such sequence of coverings can be mapped to a fixed uniformly finite sequence of coverings (albeit with enlarged parameters).
\end{proof}

Combining Theorem \ref{e1} with Propositions \ref{apropo} and \ref{fs} yields Theorem \ref{easyth}.

\section{Continuous Parseval frame on sphere}\label{S8}

In this section we construct a continuous wavelet frame on $\Sp^{d-1}$.
Unlike earlier constructions \cite{Antoine, Dahlke, Freeden, Hol, Iglewska}, our continuous wavelet frames have arbitrarily small support. We start by recalling the definition of continuous frame.

\begin{definition}
Let $\calH$ be a separable Hilbert spaces and let $(X, \nu)$ be a measure space. 
A family of vectors $\{\phi_t\}$, $t\in X$ is a {\it continuous frame} over $X$ for $\calH$
if:
\begin{itemize}

\item for each $f\in H$, the function $X\ni t \to \lan f,\phi_t \ran_{\calH} \in \C$ is measurable, and

\item  there are constants $0<A\leq B<\infty$, called {\it frame bounds}, such that
\begin{equation}\label{orety}
A\|f\|_{\calH}^2\leq \int_X |\lan f, \phi_t \ran_{\calH} |^2 d\nu \leq B\|f\|_{\calH}^2 \quad  \text{ for all } f\in \calH
\end{equation}

\end{itemize}
When $A = B$, the frame is called {\it tight}, and when $A = B = 1$, it is a
{\it continuous Parseval frame}. More generally, if only the upper bound holds in
\eqref{orety}, that is even if $A = 0$, we say that $\{\phi_t\}$, $t\in X$ is a {\it continuous Bessel family}
with bound $B$.
\end{definition}

The following elementary lemma shows the existence of a local continuous Parseval frame in $L^2(\R^k)$. For an alternative construction of a local
Parseval frame, see \cite[Theorem 4.1]{BDK}.

\begin{lemma}\label{localP}
Let $\epsilon_0>0$.
There is  a collection $\psi_t$, $t\in X$, of functions in $L^2(\R^k)$ such that:
\begin{itemize}
\item
for all $t\in X$
\[
\supp \psi_t \subset [-1-\epsilon_0,1+\epsilon_0]^{k},
\]
\item
for all $f\in L^2(\R^k)$ with $\supp f\subset [-1,1]^k$ we have
\[
\int_X |\lan f,\psi_t \ran_{L^2(\R^k)}|^2 d\nu (t)=\|  f\|_{L^2(\R^k)}^2.
\]
\end{itemize}
\end{lemma}

\begin{proof}
Take any system which is continuous Parseval frames in $L^2(\R^k)$, i.e.
for all $f\in L^2(\R^k)$  we have
\[
\int_{X} |\lan f,\psi_t \ran_{L^2(\R^k)}|^2 d\nu (t)=\|  f\|_{L^2(\R^k)}^2 
\]
Next take  a smooth function $\vp$ on $\R^k$ such that
\begin{equation}\label{multip}
\vp(x)= \begin{cases} 1 & x\in [-1,1]^k,
\\
0 & x\notin [-1-\epsilon_0,1+\epsilon_0]^k.
\end{cases}
\end{equation}
It is easy to check that $\psi_t \vp$ satisfies both conclusions of the lemma.\end{proof}

We present three examples of Parseval frames in $L^2(\R^k)$ for which Lemma \ref{localP} can be applied.

\begin{enumerate}
\item Let $E=\{0,1\}^k\setminus \{0\}$ be the non-zero vertices of the unit cube $[0,1]^k$. 
Let $\{\psi^{\mathbf e}_{j,k}: e \in E, j\in \Z, k\in \Z^k\}$ be a multivariate wavelet basis of $L^2(\R^k)$, see \cite[Proposition 5.2]{Wo}. Then, the wavelet basis is a continuous Parseval frame parameterized by $X=E \times \Z \times \Z^k$ equpped with counting measure.

\item Let $\psi\in L^2(\R^k)$ has norm one $\|\psi\|_2=1$. Then a continuous Gabor system
\[
\psi_{(t,s)}(x)=e^{2\pi i t\cdot x} \psi(x-s),\quad (t,s)\in X=\R^k\times \R^k
\]
is a continuous Parseval frame parameterized by $X$ equipped with the Lebesgue measure \cite[Corollary 11.1.4]{Chr}.

\item Let $\psi_{(x,t)}$, $(x,t)\in X=\R^k \times((0,1)\cup \{\infty\})$, be an admissible continuous wavelet introduced by Rauhut and Ullrich 
\cite[Definition 2.1]{Rauhut}.
Then for any $f\in L^2(\R^k)$ we have
\[
\int_{\R^k} \bigg( |\lan f,\psi_{(x,\infty)}\ran_{L^2(\R^k)}|^2 dx + \int_0^1 |\lan f,\psi_{(x,t)}\ran_{L^2(\R^k)}|^2 \frac{dt}{t^{k+1}} \bigg) dx=\|f\|_{L^2(\R^k)}^2.
\] 
\end{enumerate}

The concept of a local Parseval frame can be transferred to the sphere. Let 
\[
\Psi_{d-1} :[0,\pi]^{d-2}\times [0,2\pi] \to \mathbb S^{d-1}
\]
be the standard spherical coordinates given by the recurrence formula \eqref{scor}.
Fix a symmetric interior patch $\Omega$ of the form 
\[
\Omega=\Psi_{d-1}(\Theta)
\qquad\text{where } \Theta=
([\th^{d-1}_{1},\th^{d-1}_{2}]\times \cdots \times [\th^{2}_{1},\th^{2}_{2}]\times [\th^{1}_{1},\th^{1}_{2}]),
\]
where $0<\th_1^j<\th_2^j<2\pi$ for $j=1$, and $0<\th_1^j<\th_2^j<\pi$, $\th_2^j=\pi-\th_1^j$ for $j=2,\ldots,d-1$. For sufficiently small $\delta>0$, define enlargement of $\Omega$ by
\begin{equation}\label{jet}
\Omega_\delta=\Psi_{d-1}(\Theta_\delta), \qquad\text{where }\Theta_\delta:=[\th^{d-1}_{1}-\delta,\th^{d-1}_{2}+\delta]\times \cdots \times [\th^{1}_{1}-\delta,\th^{1}_{2}+\delta]).
\end{equation}

\begin{lemma}\label{localPP}
There is  a collection $\phi_t$, $t\in X$, of functions in $L^2(\Sp^{d-1})$ such that:
\begin{itemize}
\item
for all $t\in X$
\[
\supp \phi_t \subset \Omega_\delta,
\]
\item
for all $f\in L^2(\Sp^{d-1})$ with $\supp f\subset \Omega$ we have
\[
\int_X |\lan f,\phi_t \ran_{L^2(\Sp^{d-1})}|^2 d\nu (t)=\|  f\|_{L^2(\Sp^{d-1})}^2.
\]
\end{itemize}
\end{lemma}

\begin{proof}
 We use  approach from \cite{BD}, where the localized wavelet system  is transferred to the sphere  via the spherical coordinates. Consider the change of variables operator \cite[Section 6.2]{BD}
\[
\mathbf T: L^2([0,\pi]^{d-2}\times [0,2\pi]) \to L^2(\mathbb S^{d-1})
\]
given by
\[
\mathbf T(\psi)(u) =  \frac{\psi( {\Psi_{d-1}}^{-1}(u))}{\sqrt{J_{d-1}( \Psi_{d-1}^{-1}(u))}},\qquad u\in \Sp^{d-1},
\]
where $J_{d-1}$ is the Jacobian of $\Psi_{d-1}$
\[
J_{d-1}(\theta_{d-1},\theta_{d-2},\ldots,\theta_{1})=|\sin^{d-2} \theta_{d-1} \sin^{d-3} \theta_{d-2} \cdots \sin \theta_{2}|.
\]
Since the set where $\Psi_{d-1}$ is not $1$-$1$ has measure zero, by the change of variables formula, $\mathbf T$ is an isometric isomorphism.

Let $Y: \R^{d-1} \to \R^{d-1}$ be an affine transformation such that for sufficient small $\epsilon_0$
\[
Y([-1,1]^{d-1})=\Theta,\qquad Y([-1-\epsilon_0,1+\epsilon_0]^{d-1})\subset \Theta_\delta
\]
In a similar way we define  the change of variables operator ${\mathbf T_Y}$
which is an isometry
\[
 L^2([-1-\epsilon_0,1+\epsilon_0]^{d-1}) \xrightarrow{\mathbf T_Y}  L^2(\Theta_\delta)
 \]
We  transfer a local Parseval frame $\psi_t$, $t\in X$ from Lemma \ref{localP} to the sphere by isometric isomorphisms ${\mathbf T_Y}$ and $\mathbf T$
\[
 L^2([-1-\epsilon_0,1+\epsilon_0]^{d-1}) \xrightarrow{\mathbf T_Y}  L^2(\Theta_\delta)\subset L^2([0,\pi]^{d-2}\times [0,2\pi]) \xrightarrow{\mathbf T}   L^2(\mathbb S^{d-1}).
 \]
Namely, we let $\phi_t={\mathbf T}{\mathbf T_Y}\psi_t$. Then the conclusion follows from Lemma \ref{localP} since any $f\in L^2(\Sp^{d-1})$ with $\supp f\subset \Omega$ is of the form $f={\mathbf T}{\mathbf T_Y}g$ for some $g\in L^2(\R^{d-1})$ with $\supp g\subset [-1,1]^{d-1}$.
\end{proof}

\begin{theorem} 
Let $\{\phi_t\}_{t\in X}$ be a local continuous Parseval frame as in Lemma \ref{localPP}. Then, there exists a Hestenes operator $P$, which is an orthogonal projection localized on $\Omega$, such that the family
$\{ T_{b^{-1}}P\phi_t\}_{(b,t)\in SO(d)\times X}$ is a continuous Parseval frame over $( SO(d) \times X, \mu_d \times \nu)$ for $L^2(\Sp^{d-1})$.
\end{theorem}

\begin{proof}
We apply Theorem \ref{mainth} for $k=d-1$ and for a shrunk symmetric patch 
\[
\Theta_{-\delta}=[\th^{d-1}_{1}+\delta,\th^{d-1}_{2}-\delta]\times \cdots \times [\th^{1}_{1}+\delta,\th^{1}_{2}-\delta]
\]
for sufficiently small $\delta>0$. This yields a Hestenes operator $P$, which is an orthogonal projection localized on $\Omega$.
Moreover, by Proposition \ref{apropo} applied for $P$,  for any $f\in L^2(\Sp^{d-1})$ we have
\[
\begin{aligned}
 c(P) \|f\|^2= \int_{SO(d)}  \lan T_b\circ P\circ T_{b^{-1}}(f), f \ran d\mu_d(b)
& =
 \int_{SO(d)}  \lan  P T_{b^{-1}}(f), PT_{b^{-1}} f \ran d\mu_d(b) \\
 & = \int_{SO(d)}  \|  T_b \circ P\circ T_{b^{-1}}(f)\|^2 d\mu_d(b).
\end{aligned}
\]
By Lemma \ref{localPP}
\[
\int_X |\lan   f, T_{b^{-1}} P\phi_t \ran|^2 d\nu (t)
=
\int_X |\lan P T_b f,\phi_t \ran|^2 d\nu (t)=\| P T_b f\|^2 =\| T_{b^{-1}} P T_b f\|^2.
\]
Integrating the above over $SO(d)$ yields
\[
\int_{SO(d)} \int_X |\lan   f, T_{b^{-1}} P\phi_t \ran|^2 d\nu (t) d\mu_d(b) = c(P) \|f\|^2.
\qedhere
\]
\end{proof}

\bibliographystyle{amsplain}

\begin{thebibliography}{99}

\bibitem{Antoine}
J.-P. Antoine,  P. Vandergheynst, {\it 
Wavelets on the n-sphere and related manifolds.}
J. Math. Phys. 39 (1998), no. 8, 3987--4008. 

\bibitem{AWW} P. Auscher, G. Weiss, M. V. Wickerhauser,
{\it Local sine and cosine bases of Coifman and Meyer and the construction of smooth wavelets.} Wavelets, 237--256,
Wavelet Anal. Appl., 2, Academic Press, Boston, MA, 1992.




\bibitem{Bownik} M. Bownik, {\it Continuous frames and the Kadison-Singer problem.} Coherent states and their applications, 63--88, Springer Proc. Phys., 205, Springer, Cham, 2018.

\bibitem{BD}
M. Bownik, K. Dziedziul,
{\it Smooth orthogonal projections on sphere}, Const. Approx. {\bf 41} (2015), 23--48.


\bibitem{BDK} M. Bownik, K. Dziedziul, A. Kamont, {\it
    Smooth orthogonal projections on Riemannian manifold}, 2018, 
Potential Anal. 54 (2021), no. 1, 41--94.  
    
\bibitem{BDK2} M. Bownik, K. Dziedziul, A. Kamont, {\it  Parseval wavelet frames on Riemannian manifold}, J. Geom. Anal. (to appear). 

\bibitem{Chr}
O. Christensen, {\it An introduction to frames and Riesz bases.} Second edition. Applied and Numerical Harmonic Analysis. Birkh\"auser/Springer, 2016.

\bibitem{CM} R. Coifman, Y. Meyer,
{\it Remarques sur l'analyse de Fourier \`a fen\^ etre.}
C. R. Acad. Sci. Paris S\' er. I Math. {\bf 312} (1991), no. 3, 259--261.



\bibitem{Dahlke}
S. Dahlke, F. De Mari, E. De Vito, M. Hansen, M. Hasannasab, M. Quellmalz, G. Steidl, G. Teschke 
{\it Continuous Wavelet Frames on the Sphere: The Group-Theoretic Approach Revisited},   Appl. Comput. Harmon. Anal. {\bf 56} (2022), 123--149. 




\bibitem{DaiXu}
F. Dai, Y. Xu,
{\it Approximation Theory and Harmonic Analysis on Spheres and Balls.}
Springer Monographs in Mathematics. Springer, New York, 2013.

\bibitem{Diestel}
J. Diestel, J. J. Uhl, Vector measures. With a foreword by B. J. Pettis. Mathematical Surveys, No. 15. American Mathematical Society, Providence, R.I., 1977.



\bibitem{Folland}
G. Folland,  A course in abstract harmonic analysis. Second edition. Textbooks in Mathematics. CRC Press, Boca Raton, FL, 2016. 

\bibitem{Fornasier} 
M. Fornasier,  H. Rauhut, {\it 
Continuous frames, function spaces, and the discretization problem.}
J. Fourier Anal. Appl. 11 (2005), no. 3, 245--287. 

\bibitem{Freeden} 
W. Freeden, U. Windheuser, {\it Spherical wavelet transform and its discretization}, Adv. Comput. Math. 5 (1) (1996) 51--94.

\bibitem{Freeden2}
W. Freeden, U. Windheuser, {\it Combined spherical harmonic and wavelet expansion---a future concept in Earth's gravitational determination}, Appl. Comput. Harmon. Anal. 4 (1) (1997) 1--37.

\bibitem{FS} D. Freeman, D. Speegle, 
{\it The discretization problem for continuous frames},
Adv. Math. {\bf 345} (2019), 784--813. 


\bibitem{Geller} D. Geller, A. Mayeli, {\it Continuous wavelets on compact manifolds.} 
Math. Z. 262 (2009), no. 4, 895--927. 



\bibitem{HW} E. Hern\' andez, G. Weiss,
{\it A first course on wavelets.}
Studies in Advanced Mathematics. CRC Press, Boca Raton, FL, 1996.


\bibitem{Hol} M. Holschneider, {\it Continuous wavelet transforms on the sphere}, J. Math. Phys. 37 (8) (1996) 4156--4165.

\bibitem{Iglewska}
I. Iglewska-Nowak, {\it Continuous wavelet transforms on n-dimensional spheres}, Appl. Comput. Harmon. Anal. 39 (2) (2015)
248--276. 

 
 


 
\bibitem{Rauhut} 
H. Rauhut,  T. Ullrich,  {\it Generalized coorbit space theory and inhomogeneous function spaces of Besov-Lizorkin-Triebel type},
 J. Funct. Anal. 260 (2011), no. 11, 3299--3362. 
 
\bibitem{Rudin} 
  W. Rudin, {\it Functional Analysis}, McGraw Hill, 1991.
  
\bibitem{Sk0}
L. Skrzypczak, 
{\it Atomic decompositions on manifolds with bounded geometry}.
Forum Math. {\bf 10} (1998), no. 1, 19--38. 

  

\bibitem{Wo}
P.~Wojtaszczyk, {\it A mathematical introduction to wavelets}.
Cambridge University Press, Cambridge, 1997.

\bibitem{Z}
A. Zygmund, {\it Trigonometric series. Vol. I, II.}  Third edition. With a foreword by Robert A. Fefferman. Cambridge Mathematical Library. Cambridge University Press, Cambridge, 2002.


\end{thebibliography}

\end{document}